\documentclass[12pt,twoside,reqno]{amsart}
\usepackage{amsmath}
\usepackage{amsfonts}
\usepackage{amssymb}
\usepackage{color}
\usepackage{mathrsfs}
\usepackage{cite}
\usepackage{geometry}
\newcommand{\non}{\nonumber}
\newcommand{\ml}{\mathcal{L}}

\newcommand{\rd}{\mathrm{d}}
\newcommand{\mA}{\mathcal{A}}
\newcommand{\la}{\langle}
\newcommand{\ra}{\rangle}
\newtheorem{theorem}{Theorem}[section]
\newtheorem{corollary}[theorem]{Corollary}
\newtheorem{lemma}[theorem]{Lemma}
\newtheorem{proposition}[theorem]{Proposition}

\newtheorem{remark}[theorem]{Remark}
\numberwithin{equation}{section}
\begin{document}
\title{Global Existence, Uniform Boundedness, and Stabilization in a Chemotaxis System with Density-Suppressed Motility and Nutrient Consumption}
\author{Jie Jiang}
\address{Innovation Academy for Precision Measurement Science and Technology, Chinese Academy of Sciences, Wuhan 430071, HuBei Province, P.R. China}
\email{jiang@apm.ac.cn, jiang@wipm.ac.cn}

\author{Philippe Lauren\c{c}ot}
\address{Institut de Math\'ematiques de Toulouse, UMR~5219, Universit\'e de Toulouse, CNRS \\ F--31062 Toulouse Cedex~9, France}
\email{laurenco@math.univ-toulouse.fr}

\author{Yanyan Zhang}
\address{School of Mathematical Sciences, Shanghai Key Laboratory of Pure Mathematics and Mathematical Practice, East China Normal University, Shanghai 200241, P.R. China}
\email{yyzhang@math.ecnu.edu.cn}


\date{\today}

\begin{abstract}
Well-posedness and uniform-in-time boundedness of classical solutions are investigated for a three-component parabolic system which describes the dynamics of a population of cells interacting with a chemoattractant and a nutrient. The former induces a chemotactic bias in the diffusive motion of the cells and is accounted for by a density-suppressed motility. Well-posedness is first established for generic positive and non-increasing motility functions vanishing at infinity. Growth conditions on the motility function guaranteeing the uniform-in-time boundedness of solutions are next identified. Finally, for sublinearly decaying motility functions, convergence to a spatially homogeneous steady state is shown, with an exponential rate for consumption rates behaving linearly near zero.
\end{abstract}

\maketitle

%
%
\pagestyle{myheadings}
\markboth{\sc{J.~Jiang, Ph.~Lauren\c{c}ot} \& Y.~Zhang}{\sc On a Chemotaxis System with Density-Suppressed Motility and Nutrient Consumption}

\section{Introduction}\label{sec1}

Spatially periodic stripe patterns are  ubiquitous in biological systems and often play vital roles in embryogenesis and development. However, the underlying developmental mechanism remains  unclear and attracts a lot of research interest. Recently, a mathematical model for autonomous periodic stripe pattern formation was proposed in \cite{2011Science}:
\begin{equation}\label{smod0}
\begin{split}
	& \partial_t u=\Delta (u \gamma(v)) + \theta u f(n) \;\;\text{ in }\;\; (0,\infty)\times\Omega\,, \\ 	
	& \partial_t v= D_v \Delta v + \alpha u - \beta v \;\;\text{ in }\;\; (0,\infty)\times\Omega\,, \\
	& \partial_t n= D_n \Delta n - k_s \theta u f(n) \;\;\text{ in }\;\; (0,\infty)\times\Omega\,, 	
\end{split}
\end{equation}
where the consumption rate $f$ is given by
\begin{equation*}
	f(s) = \frac{s^2}{s^2+K_n}\,, \qquad s\ge 0\,.
\end{equation*}
Here, $\Omega$ is a bounded domain of $\mathbb{R}^N$, $N\ge 1$, and $u$, $v$, and $n$ denote the cell density, the signal concentration, and the nutrient level, respectively, while the cell motility $\gamma$ is a positive function and $\theta$, $D_v$, $D_n$, $k_s$, and $K_n$ are positive constants. The key feature of this model is a density-suppressed cellular motility $\gamma$, which stands for a repressive effect of the signal (and hence of the cell density) on  the cell motility. More precisely, cells perform random walks via the swim-and-tumble motion at low concentrations and are more motile while, at high concentrations, these cells tumble incessantly, a process which impedes their motion and results in a vanishing macroscopic motility. Numerical and experimental analyses indicate that such a motility control can establish a spatially periodic structure in a growing bacteria population without the recourse to other mechanisms. Subsequently, in order to better understand the effect of density-suppressed motility in pattern formations, a simplified two-component version of \eqref{smod0} with a population growth was further analyzed numerically and experimentally in \cite{2012PRL}:
\begin{equation}\label{prlmod0}
\begin{split}
	& \partial_t u =\Delta (u\gamma(v)) \;\;\text{ in }\;\; (0,\infty)\times\Omega\,,\\ 	
	& \tau\partial_t v =\Delta v+ u - \beta  v \;\;\text{ in }\;\; (0,\infty)\times\Omega\,.	
\end{split}
\end{equation}
This model can also correctly capture the dynamics at the propagating front where new stripes are formed. We remark that system~\eqref{prlmod0} belongs to the general class of chemotaxis models proposed by Keller \& Segel in their seminal work \cite{1971KS} in the 1970s, based on a local sensing mechanism for chemotaxis. It is worth pointing out  here that system~\eqref{prlmod0} is readily obtained from \eqref{smod0} by setting $\theta=0$.

As already mentioned, a key feature of chemotaxis models involving density-suppressed motility is the possible degeneracy of the diffusion in the cell's equation when the signal intensity $v$ becomes unbounded and the motility $\gamma$ vanishes at infinity. Then global well-posedness and boundedness of classical solutions to reaction-diffusion systems like \eqref{smod0} or \eqref{prlmod0} are not directly provided by classical theories for parabolic systems. The analysis of these issues have thus attracted a lot of interest in the mathematical literature recently, with a focus on the simplified two-component system~\eqref{prlmod0}. 
	
More precisely, when $N=2$, global existence and boundedness of classical solutions are first obtained in \cite{TW2017} when assuming $\gamma$ to be bounded from below by a positive constant. This assumption is relaxed in \cite{FuJi2020c} where global existence is shown for any positive and non-increasing motility function $\gamma$ converging to zero at infinity, permitting a vanishing limit at infinity. When $\tau=0$, uniform-in-time boundedness of classical solutions with a vanishing motility is  first obtained when $\gamma(s)=s^{-k}$ for any $k<\infty$ in \cite{AhnYoon2019} and later extended in \cite{FuJi2020b} to $\tau\ge 0$ and any non-increasing motility function decaying slower than any negative exponential function at infinity. Negative exponential functions appear to be critical for the dynamics of \eqref{prlmod0} and the specific choice of motility function $\gamma(s)=e^{-\chi s}$, $\chi>0$, is studied in \cite{JW2020} and  \cite{FuJi2020,FuJi2020c} by different methods, uncovering a critical mass phenomenon with the same threshold value as the classical Keller-Segel system. More precisely, classical solutions to \eqref{prlmod0} are uniformly-in-time bounded when the initial mass lies below a certain threshold value \cite{BLT2020,FuJi2020,FuJi2020c}, while unbounded global classical solutions are constructed for some initial conditions having an initial mass exceeding this threshold value \cite{FuJi2020,FuJi2020c,JW2020}. 

In higher space dimensions $N\ge 3$, global existence and boundedness of classical solutions to \eqref{prlmod0} are established when $\tau=0$ and  $\gamma(s)=s^{-k}$ with any $k<2/(N-2)$ in \cite{AhnYoon2019}, as well as in \cite{FuJi2020b,Wang20} by different approaches. As shown recently in \cite{JiLau2021a}, when $\tau=0$, global existence of classical solutions to \eqref{prlmod0} is a generic feature and is true for arbitrary positive motility functions $\gamma$ which are not necessarily monotone decreasing or decaying to zero at infinity. Uniform-in-time boundedness of classical solutions to \eqref{prlmod0} is established in \cite{JiLau2021a} when $\gamma(s)\sim s^{-k}$ as $s\rightarrow \infty$ with any $k<N/(N-2)$, a range which is likely to be optimal and extends the above mentioned papers, as well as \cite{Jiang2020} where the case $k<\min\{1,4/(N-2)\}$ is considered. When $\tau>0$, global existence of classical solutions to \eqref{prlmod0} is asserted  in \cite{FuSe2021} for any positive, non-increasing and asymptotically vanishing motility. To the best of our knowledge, uniform-in-time  boundedness is only available when $\gamma(s)=s^{-k} $ with $k\leq 1/[N/2]$ in \cite{FuJi2020b,FuJi2020c}. Here, $[N/2]$ denotes the maximal integer less or equal to $N/2$. 

To complete this overview of the existing literature on \eqref{prlmod0}, we mention that weak solutions are constructed in \cite{TW2017, LiJi2020, DKTY2019, BLT2020, DLTW2021}. Several studies have also been devoted to \eqref{prlmod0} when a logistic growth is included. The latter actually alters significantly the mathematical properties of \eqref{prlmod0} and its dissipative effect fosters boundedness of classical solutions. We refer to \cite{JKW2018,FuJi2020,WW2019} for studies in this direction.

\medskip

In contrast to its two-component counterpart~\eqref{prlmod0}, the three-component system~\eqref{smod0} has received little attention. As far as we know, global existence and boundedness of classical solutions to \eqref{smod0} are established in \cite{JSW2020} when $N=2$ and  $\gamma$ is  a bounded Lipschitz continuous function which is bounded from  below by a positive constant. Still in the two-dimensional case, global existence of classical solutions is subsequently shown in \cite{LyWa2021} when the motility $\gamma$ is a positive and decreasing function with a vanishing limit at infinity. Uniform-in-time boundedness of these solutions is then obtained provided $1/\gamma$ grows at most algebraically at infinity, a result which is consistent with what is known for \eqref{prlmod0} \cite{FuJi2020c}.
	
The aim of this paper is a thorough study of the well-posedness of \eqref{smod0} in arbitrary space dimensions, along with that of the uniform-in-time boundedness of its solutions, which not only provides a far-reaching improvement upon \cite{JSW2020,LyWa2021} but also on \cite{FuJi2020c,FuSe2021}. More precisely, on the one hand, we show the existence and uniqueness of classical solutions to \eqref{smod0} for any motility function $\gamma$ which is positive and non-increasing with a vanishing limit at infinity in any space dimension. On the other hand, uniform-in-time boundedness of classical solutions to \eqref{smod0} is obtained for motility functions decaying slower than any negative exponential at infinity in the two-dimensional case and for motility functions satisfying  $\gamma(s) \sim s^{-k}$  for an arbitrary $k<N/(N-2)$ in higher space dimension. While the former extends \cite{LyWa2021} to motility functions such as $\gamma(s) = e^{-s^\alpha}$, $\alpha\in (0,1)$, the latter improves \cite{FuJi2020c} after setting $\theta=0$ in \eqref{smod0}, since $1/[N/2]<N/(N-2)$. Moreover, when $1/\gamma$ looks like a sublinear function, we proceed partly along the lines of \cite{DLTW2021} to construct a Lyapunov functional for \eqref{smod0} which allows us to identify the long-term behavior of classical solutions to \eqref{smod0}.
 
 Before describing precisely the outcome of this paper, we recall that two different approaches have been developed in the literature to tackle the degeneracy issue while studying the existence of classical solutions to \eqref{prlmod0}. The first method relies on the derivation of a $L^\infty_tL^p_x$-estimate for $u$ for some $p>N/2$ by energy and duality methods, which gives rise to an $L^\infty_{t,x}$-estimate for $v$, the latter being deduced from the equation for $v$ due to standard regularity theory for parabolic/elliptic equations. However, this method needs restrictive assumptions on $\gamma$, see, e.g., \cite{TW2017, AhnYoon2019, Wang20,JW2020}. The other approach is based on  a two-step comparison argument proposed in \cite{FuJi2020,FuJi2020c}. The key ingredient lies in the introduction of an intermediate auxiliary function $w$ which is the solution to a linear elliptic equation. Owing to the specific structure of the cell's equation, suitable applications of elliptic and parabolic comparison principles allow one to derive an upper bound on $w$, as well as a control from above on $v$ by $w$. An upper bound for $v$ thus  follows and subsequently leads to the global existence of classical solutions in any space dimension for rather generic motility functions \cite{JiLau2021a,FuSe2021}. This approach also turns out to be particularly efficient to investigate the uniform-in-time boundedness issue but requires a more refined argument, as developed in \cite{FuJi2020b,Jiang2020,JiLau2021a}. In fact, an important intermediate step is the derivation of an evolution equation for $w$ with a source term growing sublinearly with $w$. The analysis performed in this paper is actually based on a further development of this second method.

\medskip

Coming back to \eqref{smod0}, we focus here on the global well-posedness of \eqref{smod0} for a general class of motility functions $\gamma$ and consumption rates $f$, as well as on the uniform-in-time boundedness of its solutions. In addition, we shall investigate the long-term behavior of its solutions for a suitable class of motility functions $\gamma$. To reduce the number of parameters in \eqref{smod0}, we first perform the  rescaling
\begin{align*}
	& \bar{t} = \tau t\,, \quad \bar{x} = x/L\,, \quad (u,v,n)(t,x) = (\bar{u},H \bar{v},M \bar{n})(\bar{t},\bar{x})\,, \\
	& \bar{\gamma}(s) = \gamma(Hs)/(\tau L^2)\,, \quad \bar{f}(s) = \theta f(Ms)/\tau\,,
\end{align*}
with 
\begin{equation*}
	\tau = D_n/D_v\,, \quad L = \sqrt{D_v}\,, \quad H = \alpha\,, \quad M = k_s\,.
\end{equation*}
Dropping the bars, the system \eqref{smod0} becomes
\begin{subequations}\label{cp}
	\begin{align}
		& \partial_t u  = \Delta (u \gamma(v)) + u f(n)\,, \qquad (t,x)\in (0,\infty)\times\Omega\,, \label{cp1} \\
		&	\tau \partial_t v=\Delta v - \beta v  +  u\,, \qquad (t,x)\in (0,\infty)\times\Omega\,, \label{cp2} \\
		& \partial_t n= \Delta n - u f(n) \,, \qquad (t,x)\in (0,\infty)\times\Omega\,, \label{cpn} \\
		& \nabla (u\gamma(v)) \cdot \nu = \nabla v \cdot \nu =\nabla n \cdot \nu = 0\,, \qquad (t,x)\in (0,\infty)\times\partial\Omega\,, \label{cp3} \\
		& (u,v,n)(0)  = \left(u^{in},v^{in}, n^{in}\right)\,, \qquad x\in\Omega\,, \label{cp4}
	\end{align}
\end{subequations}
which only involves two positive parameters $\tau>0$ and $\beta>0$. Here,  $\Omega$ is a bounded domain of $\mathbb{R}^N$ ($N\ge 1$) with smooth boundary. The motion of cells is biased by the local concentration $v$ of chemotactic signal and is prescribed by the motility $\gamma(v)$ of cells, which is a positive function of $v$. Recall that, due to the density-suppressed effect, $\gamma$ is a non-increasing function on $(0,\infty)$. The function $f$ is non-negative and represents the consumption rate of nutrients by cells, which generalizes the particular form given in \eqref{smod0}.

\medskip

To begin with, we introduce some basic assumptions and notations. Throughout this paper we use the short notation $\|\cdot\|_{p}$ for the norm $\|\cdot\|_{L^p(\Omega)}$ with  $p\in[1,\infty]$. For the initial condition $(u^{in}, v^{in}, n^{in})$, we require that
\begin{equation}\label{ini}
	\begin{split}
& \left( u^{in},v^{in},n^{in} \right)\in   W^{1,N+1}(\Omega;\mathbb{R}^3)\,, \quad u^{in}\not\equiv 0\,, \\
& u^{in}\ge 0\,, \quad n^{in}\ge 0\,, \quad v^{in}>0 \quad  \mbox{in } \bar\Omega\,.
	\end{split}
\end{equation}
The motility function $\gamma$ and consumption rate $f$ are assumed to satisfy
\begin{equation}
	\gamma\in C^3((0,\infty))\,, \quad \gamma>0\,, \quad \gamma'\leq 0 \;\;\text{ in  }\;\; (0,\infty)\,, \quad\lim\limits_{s\rightarrow\infty}\gamma(s)=0\,, \tag{A1} \label{g1}
\end{equation}
and
\begin{equation}\label{asf}
	f\in C^{1}([0,\infty))\,, \quad f(0)=0\;\; \text{ and }\;\; f\ge 0\;\;\text{ on }\;\; (0,\infty)\,,
\end{equation}
respectively. Clearly, $f\equiv 0$ satisfies \eqref{asf}, so that the results obtained below equally apply to the two-component system~\eqref{prlmod0}.

We are now in a position to state our first main result concerning global existence of classical solutions to \eqref{cp}.

\begin{theorem}\label{TH1} Let $N\ge 1$. Suppose that $\gamma$ and $f$ satisfy assumptions~\eqref{g1} and~\eqref{asf}, respectively, and that the initial condition $(u^{in}, v^{in}, n^{in})$ satisfies \eqref{ini}. Then problem~\eqref{cp} has a unique global non-negative classical solution $(u,v,n)\in C([0,\infty)\times \bar{\Omega};\mathbb{R}^3) \cap C^{1,2}((0,\infty)\times \bar{\Omega};\mathbb{R}^3)$.
\end{theorem}

 Observe that \eqref{g1} does not require $\gamma$ to be bounded as $s\to 0$, so that our analysis includes in particular $\gamma(s)=s^{-k}$ for $k>0$. This feature is actually not surprising in view of \eqref{e00} below, which states that $v$ has a time-independent  positive lower bound $v_*$, which is also independent of the choice of $\gamma$. Also, since no growth condition is required on $\gamma$ in Theorem~\ref{TH1}, a noticeable outcome of Theorem~\ref{TH0} is that the density-suppressed motility plays a fundamental role in preventing finite time blowup, which is in sharp contrast with the classical Keller-Segel system.
 
\medskip

We next investigate the boundedness of classical solutions to \eqref{cp} under certain decay assumptions of $\gamma$ at infinity and begin with the two-dimensional case $N=2$.

\begin{theorem}\label{TH0}
Assume $N=2$ and consider $f$ satisfying \eqref{asf}. Suppose that $\gamma$ satisfies assumption~\eqref{g1} and that there is $\chi>0$ such that
\begin{equation}\label{A2b}
	\liminf\limits_{s\rightarrow\infty}e^{\chi s}\gamma(s)>0\,. \tag{A2$\mathrm{e}_\chi$}
\end{equation}
If the initial condition $(u^{in},v^{in},n^{in})$ satisfies \eqref{ini} with 
\begin{equation}
	\|u^{in}+n^{in}\|_{1}<\frac{4\pi}{\chi}\,, \label{massini}
\end{equation}
then the global classical solution $(u,v,n)$ to \eqref{cp} is uniformly-in-time bounded; that is,
\begin{equation}
	\sup_{t\ge 0}\left\{ \|u(t)\|_\infty + \|v(t)\|_{\infty} + \|n(t)\|_\infty \right\} < \infty\,, \label{ubd}
\end{equation}
and, for any $t_0>0$,
\begin{equation}
	\sup_{t\ge t_0}\left\{ \|u(t)\|_{C^1(\bar{\Omega})} + \|v(t)\|_{C^1(\bar{\Omega})} + \|n(t)\|_{C^1(\bar{\Omega})} \right\} < \infty\,. \label{lipbd}
\end{equation} 

In particular, if $\gamma$ satisfies \eqref{g1} and \eqref{A2b} for all $\chi>0$, then the global classical solution to \eqref{cp} is uniformly-in-time bounded for any initial condition $(u^{in},v^{in},n^{in})$ satisfying \eqref{ini}.
\end{theorem}

A similar result is obtained in \cite{LyWa2021} under the stronger assumption that $1/\gamma$ decays algebraically at infinity. Theorem~\ref{TH0} shows that this assumption can be relaxed and applies in particular to $\gamma(s)=e^{-\beta s^{\theta}}$ with any $\beta>0$ and $\theta\in(0,1)$.
	
\begin{remark}\label{2Doptimal}
	The boundedness results in Theorem~\ref{TH0} are optimal. Indeed, when $\gamma(s)=e^{-\chi s}$ and $f\equiv0$ (so that \eqref{cp} reduces to \eqref{prlmod0}), a critical mass phenomenon is observed in \cite{FuJi2020c,JW2020}. Uniform-in-time boundedness is obtained when the total mass of cells is less than the critical value $4\pi/\chi$.  Moreover, finite-time blowup is excluded in \cite{FuJi2020c,BLT2020} by different methods and initial conditions having a sufficiently large total mass of cells which lead to unbounded solutions to \eqref{cp} (with $f\equiv 0$) are constructed in \cite{FuJi2020c,JW2020}.
\end{remark}
 
 \medskip
 
 We turn to higher space dimensions $N\ge 3$ and first point out that the assumption~\eqref{A2b} is not sufficient to guarantee uniform-in-time boundedness of global classical solutions. Indeed, it is shown in \cite{FuSe2021} that, for an arbitrarily given $m>0$, there always exist radially symmetric initial conditions $(u^{in},v^{in})$ with $m=\|u^{in}\|_1$ in the ball $\Omega = B_R(0)$ such that the corresponding global classical solution to \eqref{cp} with $\gamma(s)=e^{-s}$ and $f\equiv 0$ is unbounded. It is thus likely that uniform-in-time boundedness requires more restrictive growth conditions on $1/\gamma$ at infinity. Specifically, the following algebraic growth assumption is requested on $1/\gamma$:
\begin{equation}\label{gamma2}
\text{there are $k\geq l\geq0$ such that}\;\;\liminf\limits_{s\rightarrow\infty}s^{k}\gamma(s)>0\;\;\text{and}\;\;
\limsup\limits_{s\rightarrow\infty}s^{l}\gamma(s)<\infty. \tag{A2}
\end{equation}

We assume in addition that
\begin{equation}\label{A3}
	\begin{split}
	&\text{there is}\;  b_0\in (0,1] \;\text{such that, for any $s\ge s_0>0$,} \\
	& \hspace{2cm} s\gamma(s)+(b_0-1) \int_1^s\gamma(\eta)\mathrm{d}\eta\leq K_0(s_0) \,, \\
	&\text{where $K_0(s_0)>0$ depends only on $\gamma$, $b_0$, and $s_0$}\,.
	\end{split}\tag{A3}
\end{equation}

A suitable choice of the exponents $k$ and $l$ in \eqref{gamma2} leads us to the boundedness of global solutions to \eqref{cp} in higher dimensions, which we state now.

\begin{theorem}\label{TH2}
 Let $N\ge 3$ and consider $f$ satisfying \eqref{asf}. Suppose that $\gamma$ satisfies assumption~\eqref{g1} and
 assumption~\eqref{gamma2} for some $k\geq l\geq 0$ satisfying 
\begin{equation*}
 	k < \frac{N}{N-2} \;\;\text{and }\;\; k-l < \frac{2}{N-2}\,.
\end{equation*}
Assume further that, either $\gamma$ satisfies assumption~\eqref{A3}, or the parameter $l$ in \eqref{gamma2} is such that $l>\frac{(N-4)_+}{N-2}$.  Then, for any initial condition $(u^{in},v^{in},n^{in})$ satisfying  \eqref{ini}, the  global classical solution $(u,v,n)$  to \eqref{cp} is uniformly-in-time bounded in the sense that it satisfies \eqref{ubd} and \eqref{lipbd}.
\end{theorem}

The assumption~\eqref{A3} is somewhat a technical one and is satisfied by several typical examples of motility functions $\gamma$, including $(a_1+s)^{-k_1}$, $(a_1+s)^{-k_1}\log^{-k_2} (a_2+s)$, and $(a_1+s)^{-k_1}+(a_2+s)^{-k_2}$, with  $a_i\ge 0$ and $k_i>0$ ($i=1,2$), which also satisfy \eqref{g1} and \eqref{A2b}. We emphasize that, when $\gamma$ satisfies \eqref{gamma2} with $l\geq 1$, then $\gamma(s)\le C s^{-l}$ for $s$ large enough, which in turn guarantees that $s\gamma(s)\le C$ and the validity of \eqref{A3} with $b_0=1$. We refer to Lemma~\ref{sufconA3} below for a more detailed discussion.

A direct consequence of Theorem~\ref{TH2} is that, when $N\ge 3$ and $\gamma(s)\sim s^{-k}$ as $s\rightarrow\infty$ for some  $0<k<N/(N-2)$, global classical solutions to \eqref{cp} are uniformly-in-time bounded. In view of what is already known in the two-dimensional case, the main question left open is the optimality of the exponent $N/(N-2)$ and it will be our future task to study whether the exponent $N/(N-2)$ is critical for boundedness of solutions.

\begin{remark}\label{RemTCS}
It is worth mentioning once more that, given non-negative functions $(u^{in},v^{in}, n^{in})\in W^{1,N+1}(\Omega;\mathbb{R}^3)$ such that $u^{in}\not \equiv 0$, $v^{in}>0$ in $\bar{\Omega}$ and $n^{in}\equiv 0$, the corresponding solution $(u,v,n)$ to \eqref{cp} satisfies $n\equiv 0$ due to $f(0)=0$ and $(u,v)$ is actually a solution to \eqref{prlmod0}. An interesting consequence of this property is that, when $\gamma$ satisfies \eqref{g1}, global existence of solutions to \eqref{prlmod0} follows from Theorem~\ref{TH1} and we thus provide an alternative proof of \cite[Theorem~1.1]{FuSe2021}. Assuming further that $\gamma$ satisfies the assumptions of Theorem~\ref{TH2}, we also deduce from Theorem~\ref{TH2} the uniform-in-time boundedness of classical solutions to \eqref{prlmod0}, thereby extending \cite{FuJi2020b, FuJi2020c}.
\end{remark}
\medskip

Our final result deals with the large time behavior of globally bounded solutions to \eqref{cp} when $1/\gamma$ grows at most linearly at infinity.

\begin{theorem}\label{TH3}
Suppose that $f$ satisfies \eqref{asf} with $f>0$ on $(0,\infty)$ and that the initial condition satisfies \eqref{ini}.  Assume that  $\gamma$ satisfies \eqref{g1} and 
\begin{equation}\label{g_1'}
	s\gamma'(s)+\gamma(s)\geq0\,, \qquad s\in (0,\infty)\,.
\end{equation}
If  $(u,v,n)$ is a non-negative classical solution to \eqref{cp} which is uniformly-in-time bounded, then 
\begin{equation*}
	\lim\limits_{t\rightarrow\infty} \left( \|u(t)-m\|_{\infty} + \|v(t)-m\|_{\infty} + \|n(t)\|_{\infty} \right) = 0\,,
\end{equation*}
where $m\triangleq \|u^{in}+n^{in}\|_1/|\Omega|$.

In addition, if $f$ satisfies
\begin{equation}\label{asf002}
		\liminf\limits_{s\rightarrow0+}f(s)/s>0,
\end{equation}
then there exist $\delta_0>0$ and $C>0$ depending on $\Omega$, $\gamma$, $f$,  and the initial data such that
\begin{equation*}
	\|u(t)-m\|_{\infty}+\|v(t)-m\|_{\infty}+\|n(t)\|_{\infty}\leq Ce^{-\delta_0 t}\qquad\text{for all}\;\; t\geq0.
\end{equation*}
\end{theorem}

A first consequence of Theorem~\ref{TH3} is that there is no pattern formation in the dynamics of \eqref{cp} when $\gamma$ satisfies \eqref{g1} and \eqref{g_1'}, as already observed in \cite{AhnYoon2019, DLTW2021, Jiang2020}. Since the latter somehow means that $\gamma(s)\sim C s^{-l}$ as $s\to\infty$ for some $l\in (0,1]$ and thus that $\gamma$ decays rather slowly at infinity, pattern formation in \eqref{cp} can only be triggered by a motility function which decreases sufficiently rapidly at infinity, an observation which fully complies with \cite{2011Science}. 

The proof of Theorem~\ref{TH3} is divided into two steps: we first show that the Lyapunov functional constructed in \cite{DLTW2021} for the two-component system~\eqref{prlmod0} can be modified in a suitable way (with terms involving $n$ in particular) to give rise to a Lyapunov functional for \eqref{cp}. It is worth pointing out that the Lyapunov functional constructed here (and in \cite{DLTW2021}) is different from that obtained in \cite{AhnYoon2019} when $\tau=0$. The second step is devoted to the derivation of a lower bound of the dissipation of the Lyapunov functional in terms of the Lyapunov functional itself, which eventually leads to the exponential decay.

\medskip

As a consequence of the above results, we have the following result for motility functions which are negative power laws.

\begin{proposition}
	Let $N\ge 2$ and assume that $\gamma(s)=s^{-k}$ for some $k>0$ and that $f$ satisfies \eqref{asf}. For any initial condition $(u^{in},v^{in},n^{in})$ satisfying \eqref{ini}, problem~\eqref{cp} has a global classical solution. Moreover, it is uniformly-in-time bounded when, either $k<\infty$ and $N=2$, or $k<N/(N-2)$ and $N\geq 3$. In addition, if $0<k\le 1$ and $f>0$ on $(0,\infty)$, then the global classical solution to \eqref{cp} converges to $(m,m,0)$ as time goes to infinity, and the convergence takes place at an exponential rate provided that $f$ satisfies \eqref{asf002}.
\end{proposition}

The main idea of our proof relies on an improved comparison method. In fact, we not only need a control of $v$ from above by a suitably defined auxiliary function via the two-step comparison argument originally developed in \cite{FuJi2020c}, but also require a reverse control of $v$ from below. The latter goal is achieved by a third step comparison argument developed in the present contribution. In order to illustrate our strategy more explicitly, we restrict ourselves to the case $f\equiv 0$, which actually corresponds to the simplified two-component system~\eqref{prlmod0}. Later, we will explain how to extend the argument to the more involved three-component case.

 To begin with, we introduce the auxiliary function $w=\mA^{-1}u$, where $\mA$ denotes the elliptic operator $-\Delta+\beta$ in $\Omega$ with homogeneous Neumann boundary conditions on $\partial\Omega$. Applying $\mA^{-1}$ to both sides of  \eqref{cp1} gives the following key identity:
\begin{equation}\label{keyid0}
	\partial_t w+u\gamma(v)= \beta\mA^{-1}[u\gamma(v)]\,,
\end{equation}
which unveils the intrinsic mechanism of the density-suppressed motility and the nonlinear coupling structure. In fact, thanks to the monotonicity of $\gamma$ and the elliptic comparison principle, one observes that
\begin{equation}\label{nonloc0}
	0 \le \mA^{-1}[u\gamma(v)]\leq \mA^{-1}[\gamma(v_*)u]=\gamma(v_*)w,
\end{equation}
where $v_*$ is a time-independent  positive lower bound for $v$ given in \eqref{e00} below. Integrating the key identity \eqref{keyid0} and using the non-negativity of $u\gamma(v)$ entail that
\begin{equation}\label{upcon0}
	0 \le w(t,x)\leq Ce^{\beta\gamma(v_*)t}\,, \qquad (t,x)\in [0,\infty)\times \Omega\,.
\end{equation}
We remark that, if $\tau=0$, then $w=v$ and the bound~\eqref{upcon0} actually shows the boundedness of $v$. However, when $\tau>0$, there holds $w = v + \tau\mA^{-1}[\partial_t v]$ and deriving a comparison between $w$ and $v$ is more involved and requires additional arguments. Specifically, introducing the parabolic operator $\ml  z= \tau \partial_t z - \Delta z + \beta z$ and a function $\Gamma$ satisfying $\Gamma'=\gamma$, the monotonicity of $\gamma$ and the key identity \eqref{keyid0} allow us to show by delicate calculations that $\ml v\leq \ml (w+\Gamma(v)+K)$ for some constant $K>0$. Then we may employ the parabolic comparison principle to get $v\leq w+\Gamma(v)+K$. Next, the monotonicity and the vanishing limit of $\gamma$ imply that $\Gamma(s)\leq \alpha s$ for all $s>0$ and some $0<\alpha<1$. As a result, we obtain an upper control of $v$ by $w$: 
\begin{equation}\label{upcontrol}
	v\leq \frac{w+K}{1-\alpha}\,,
\end{equation}
which, together with the previously obtained upper bound~\eqref{upcon0}, gives finally an upper bound for $v$.

The upper bound~\eqref{upcontrol} plays a key role in the previous studies performed in \cite{FuJi2020b,FuSe2021,FuJi2020c,JiLau2021a} concerning existence and boundedness in the two-component system~\eqref{prlmod0}. However, it does not provide the boundedness of classical solutions to \eqref{prlmod0} in higher dimensions $N\geq 3$ when $\tau>0$. To better highlight the difficulty to be overcome, let us recall that the strategy set up in \cite{JiLau2021a} to prove boundedness of classical solutions to \eqref{prlmod0} when $\tau=0$ is to regard the key identity~\eqref{keyid0} as a quasilinear parabolic equation for $v$ with a non-local source term. Indeed, since $w=v$ when $\tau=0$, an alternative form of \eqref{keyid0} is:
 \begin{equation}\label{keyid1}
	\partial_t v-\gamma(v)\Delta v+ \beta v\gamma(v)= \beta\mA^{-1}[u\gamma(v)].
\end{equation} 
Then,  using once more the elliptic comparison principle and the monotonicity of $\gamma$, one may manipulate the non-local term to derive the following inequality
\begin{equation}\label{comp1}
	\mA^{-1}[u\gamma(v)]\leq \Gamma(v)+C\,.
\end{equation}
At this stage, one may further notice that, in the special case $\gamma(s)\sim s^{-k}$ as $s\rightarrow\infty$, the function $\Gamma$ is controlled  at infinity by $C s^{(1-k)_+} $ when $k\neq1$, or by $Cs^{\varepsilon}$ for any $\varepsilon>0$ when $k=1$, see Lemma~\ref{upG} below. Based upon this key observation, a delicate iterative argument is applied to \eqref{keyid1} to establish a time-independent  upper bound for $v$ provided that $k<N/(N-2)$, see \cite{JiLau2021a}. In order to emphasize the need of \eqref{comp1}, we mention that a control of the non-local term $\mA^{-1}[u\gamma(v)]$  by a linear function of $v$ as done in \eqref{nonloc0} only allows one to derive a uniform-in-time upper bound on $v$ when, either $k<2$ and $N=3$, or $k\in(0, 1]\cap(0,4/(N-2))$ and $N\geq 4$, see \cite{Jiang2020}.

In contrast, when $\tau>0$, the situation becomes rather involved, as $w\ne v$ and the  alternative form of the key identity~\eqref{keyid0}, which  reads
\begin{equation}\label{keyid2}
	\partial_t w-\gamma(v)\Delta w + \beta w\gamma(v)  = \beta\mA^{-1}[u\gamma(v)]\,,
\end{equation}
in that case, features both $v$ and $w$. In order to derive a single equation for $w$, we aim at replacing $\gamma(v)$ by $\gamma(w)$ in the above identity. Owing to the monotonicity of $\gamma$ and the non-negativity of $-\Delta w + \beta w = u$, a two-sided control of $v$ by $w$ will do the job. For this purpose, we develop an additional argument to establish the reverse estimate $w\le C(v+1)$. Together with \eqref{upcontrol} and positive lower bounds on $v$ and $w$ which are derived independently, we finally arrive at the two-sided control $C_1 w\leq v\leq C_2 w$, which in turn implies that 
\begin{equation}
	u \gamma(C_2 w) \le u\gamma(v)\leq u\gamma(C_1 w)\,, \label{ucon}
\end{equation}
due to the monotonicity of $\gamma$. Recalling that $u = \mA[w]$, we deduce from \eqref{keyid2} and \eqref{ucon} that
\begin{equation}
	\partial_t w-\gamma(C_2w)\Delta w + \beta \gamma(C_2w)w \leq \beta\mA^{-1}[u\gamma(C_1w)]\,. \label{keyineq}
\end{equation}
Now, \eqref{keyineq} looks very much like \eqref{keyid1} and we may proceed as in the derivation of \eqref{comp1} to estimate the non-local term $\mA^{-1}[u\gamma(C_1w)]$ by $\Gamma(w)+C$. We are then able to carry out an iterative Moser technique to get a uniform-in-time upper bound for $w$ provided that $k<N/(N-2)$. Then a time-independent upper bound of $v$ follows as well.

Having obtained the uniform-in-time boundedness of $v$, we can further show that $w$ and $v$ are H\"{o}lder-continuous with respect to both $t$ and $x$ by establishing a local energy estimate, following a classical approach developed in \cite{LSU1968}. Thanks to this property, we can employ the theory developed by Amann in \cite{Aman1988, Aman1989, Aman1990, Aman1993, Aman1995} to find a representation formula for $w$, which involves a parabolic evolution operator having properties similar to an analytic semigroup. The estimates for the parabolic evolution operator given in \cite{Aman1995} enable us to derive the (time-independent) $W^{1,\infty}$-estimates of $w$ and $v$, which in turn give rise to the uniform-in-time boundedness of $u$.

Let us finally point out that a careful modification of the above mentioned arguments is needed when dealing with the three-component system~\eqref{cp}. In fact, since it is no longer the total mass of $u$ which is invariant throughout time evolution but that of $u+n$, a second auxiliary function $S \triangleq \mA^{-1}[u+n]$ is introduced. While only an upper bound on $v$ in terms of $S$ is needed to establish global existence as stated in Theorem~\ref{TH1}, the uniform-in-time boundedness reported in Theorem~\ref{TH2} requires more work and a two-sided control of $v$  by $S$ is in fact needed in the proof.

\medskip

The remainder of the paper is organized as follows.  In Section~\ref{sec2}, we provide some preliminary results and recall some useful lemmas. In Section~\ref{sec-comparison}, we revisit the comparison argument and develop an additional argument to derive the above mentioned two-sided control. In Section~\ref{gecs}, we study the global existence of classical solutions to \eqref{cp}. In Section~\ref{sec-2D}, we prove uniform-in-time boundedness of classical solutions in the two-dimensional case $N=2$, while the case of higher space dimension $N\ge 3$ is dealt with in Section~\ref{ubcs}. The last section is devoted to the large time behavior of bounded classical solutions to \eqref{cp} when $\gamma$ satisfies \eqref{g1} and \eqref{g_1'} and $f$ satisfies \eqref{asf} and \eqref{asf002}.

\section{Preliminaries}\label{sec2}
%
\newcounter{NumConst}

In this section, we recall some useful results. We begin with the existence of local classical solutions which mainly follows from the theory developed by Amann in \cite{Aman1988, Aman1989, Aman1990, Aman1993} and the comparison principle, along with positivity properties of the heat equation.

 \begin{theorem}\label{local}
	 Suppose that $\gamma$ satisfies \eqref{g1} and $(u^{in},v^{in},n^{in})$ satisfies \eqref{ini}. Then there exists $T_{\mathrm{max}} \in (0, \infty]$ such that problem~\eqref{cp} has a unique non-negative classical solution $(u,v,n)\in C([0,T_{\mathrm{max}})\times \bar{\Omega};\mathbb{R}^3)\cap C^{1,2}((0,T_{\mathrm{max}})\times \bar{\Omega};\mathbb{R}^3)$. The solution $(u,v,n)$ satisfies the mass conservation
	\begin{equation}
	\int_\Omega \left(u(t,x)+n(t,x)\right)\ \mathrm{d} x= \int_\Omega (u^{in}(x)+n^{in}(x))\ \mathrm{d} x
	\quad \text{for\ all}\ t \in (0,T_{\mathrm{max}})\,,\label{e0}
	\end{equation}
and 
\begin{equation}
	\|u^{in}+n^{in}\|_1 \ge \|u(t)\|_{1}\ge \|u^{in}\|_{1}>0\quad \text{for\ all}\ t \in (0,T_{\mathrm{max}}).\label{uL1}
\end{equation}		
Moreover, for any $1\leq p\leq \infty$, there holds
\begin{equation}
	\|n(t)\|_{p}\leq \|n^{in}\|_{p}\quad \text{for\ all}\ t \in (0,T_{\mathrm{max}}),\label{nup}
\end{equation}
and there is $v_*>0$ depending only on $\Omega$, $v^{in}$, and $\|u^{in}\|_1$ such that
\begin{equation}
	v(t,x) \ge v_*\,, \qquad (t,x)\in [0,T_{\mathrm{max}})\times \bar{\Omega}\,. \label{e00}
\end{equation}
Finally, if $T_{\mathrm{max}}<\infty$, then
	\begin{equation*}
	\limsup\limits_{t\nearrow T_{\mathrm{max}}}\|u(t)\|_{\infty}=\infty.
	\end{equation*}
\end{theorem}

\begin{proof}
	We set $D_0 = \mathbb{R}\times (0,\infty)\times \mathbb{R}$, so that $\left( u^{in},v^{in},n^{in} \right)$ ranges in $D_0$ according to \eqref{ini}. Owing to the regularity \eqref{g1}, \eqref{asf}, and \eqref{ini} of $\gamma$, $f$, and the initial conditions, we infer from \cite[Theorems~14.4 \&~14.6] {Aman1993} (with $s=1$ and $p=N+1$) that there is a unique classical solution $(u,v,n)\in C([0,T_{\mathrm{max}})\times \bar{\Omega};D_0)\cap  C^{1,2}((0,T_{\mathrm{max}})\times \bar{\Omega};\mathbb{R}^3)$ to \eqref{cp} and that, if $T_{\mathrm{max}}<\infty$, then
	\begin{equation*}
	\limsup\limits_{t\nearrow T_{\mathrm{max}}} \left( \|u(t)\|_{\infty} + \|v(t)\|_{\infty} + \left\|\frac{1}{v(t)} \right\|_{\infty} + \|n(t)\|_{\infty} \right) =\infty.
\end{equation*}
	First, \eqref{cp1}, \eqref{cpn}, \eqref{asf}, and the comparison principle guarantee that
	\begin{equation*}
		u(t,x)\ge 0\,, \quad \|n^{in}\|_\infty \ge n(t,x) \ge 0\,, \qquad (t,x)\in  [0,T_{\mathrm{max}})\times \bar{\Omega}\,.
	\end{equation*}
Using again the non-negativity of $f$, the just established non-negativity of $u$ and $n$ allows us to  derive the mass conservation \eqref{e0}, the upper and lower bounds \eqref{uL1}, and the integrability estimates \eqref{nup} on $n$. Next, owing to the lower bound \eqref{uL1} on $\|u\|_1$ and the assumed positivity of $v^{in}$, we deduce the lower bound \eqref{e00} from \cite[Lemma~2.6]{Fuji2016}. 

Assume finally that $T_{\mathrm{max}}<\infty$. According to \eqref{nup} and \eqref{e00}, $\|n\|_\infty$ and $\|1/v\|_\infty$ cannot blow up, while \eqref{cp2} and the comparison principle imply that $\|v\|_\infty\le \max\{ \|v^{in}\|_\infty , \|u\|_\infty/\beta\}$. Consequently, $\|u(t)\|_\infty$ has to be unbounded as $t\to T_{\mathrm{max}}$ and the proof is complete.
\end{proof}

Throughout this paper,  we set 
\begin{equation*}
	\gamma^* \triangleq \sup\limits_{s\ge v_*} \gamma(s)\,, \qquad  f^* \triangleq\sup\limits_{0\leq s\leq \|n^{in}\|_{\infty}}f(s)<\infty.
\end{equation*}
Then, in view of \eqref{nup} and \eqref{e00}, there holds
\begin{equation}\label{gvup}
		0< \gamma(v)\leq \gamma^* \;\;\text{for\ all}\ (t,x) \in
		[0,T_{\mathrm{max}})\times\bar{\Omega}\,,
\end{equation}
and
\begin{equation}\label{fup}
	0\leq f(n)\leq f^* \;\;\text{for\ all}\ (t,x) \in [0,T_{\mathrm{max}})\times\bar{\Omega}\,.
\end{equation}

Next, we let $(\cdot)_+=\max\{\cdot,0\}$ and recall the following result,  see \cite[Proposition~(9.2)]{Aman1983}, \cite[Lemme~3.17]{BeBo1999}, or \cite[Lemma~2.2]{AhnYoon2019}.

\begin{lemma}\label{lm2} \refstepcounter{NumConst}\label{cst1}
	Let $f\in  L^1(\Omega)$. For any $1\leq q< \frac{N}{(N-2)_+}$, there exists a positive constant $C_{\ref{cst1}}(q)$ depending only on $\Omega$ and $\beta$ such that the solution $z \in W^{1,1}(\Omega)$ to
	\begin{equation}\label{helm}
	\begin{cases}
	-\Delta z+\beta z=f,\qquad x\in\Omega\,,\\
	\nabla z\cdot \nu=0\,,\qquad x\in\partial\Omega\,,
	\end{cases}
	\end{equation}
	 satisfies
	\begin{equation*}
	\|z\|_q \leq C_{\ref{cst1}}(q) \|f\|_1\,.
	\end{equation*}
\end{lemma}

When $N=2$, we need the following result given in \cite[Lemma~3.3]{Wang20}, which is similar to the celebrated Brezis-Merle inequality \cite[Theorem~1]{BrMe1991}, see \cite[Proposition~6.1]{NSS2001} and \cite[Lemma~A.3]{TW2017} for related results. 

\begin{lemma}\label{lm2e} \refstepcounter{NumConst}\label{cst2}
	Assume that $N=2$. For any  $f\in L^1(\Omega)$ such that
	\begin{equation*}
		\|f\|_1= \Lambda>0
	\end{equation*}
	and $0<R<\frac{4\pi}{\Lambda}$, there is $C_{\ref{cst2}}(\Lambda,R)>0$ depending on $\Omega$, $\beta$, $\Lambda$, and $R$ such that the solution $z$ to \eqref{helm} satisfies
	\begin{equation*}
		\int_\Omega e^{Rz}\rd x\leq C_{\ref{cst2}}(\Lambda,R)\,.
	\end{equation*}
\end{lemma}

We next define 
\begin{equation}\label{defG}
\Gamma(s)=\int_{1}^s\gamma(\eta)\mathrm{d}\eta\,, \qquad s>0\,,
\end{equation}
and  recall the following simple relation between $\gamma$ and $\Gamma$  derived in \cite[Lemma~6]{FuJi2020c}.

\begin{lemma}\label{lemGam} \refstepcounter{NumConst}\label{cst3}
Assume  \eqref{g1} and let $\varepsilon>0$.  There is $C_{\ref{cst3}}(\varepsilon)>0$ depending on $\varepsilon$ and $\gamma$ such that
\begin{equation}\label{Gamma0}
s\gamma(s)-\gamma(s_0)\leq \Gamma(s)\leq \varepsilon s + C_{\ref{cst3}}(\varepsilon)\,, \qquad s\geq s_0>0\,.
\end{equation}
\end{lemma}

\begin{proof} Let $\varepsilon>0$. Owing to \eqref{g1}, there is $s_\varepsilon>1$ such that $\gamma(s) \le \varepsilon$ for $s\ge s_\varepsilon$. Consequently, if $s\ge 1$, then
\begin{equation*}
	\Gamma(s)  = \int_1^{s_\varepsilon} \gamma(\eta)\ \mathrm{d}\eta + \int_{s_\varepsilon}^s \gamma(\eta)\ \mathrm{d}\eta \le s_\varepsilon \gamma(1) + \varepsilon (s-s_\varepsilon)_+ \le s_\varepsilon \gamma(1) + \varepsilon s\,,
\end{equation*}	
while $\Gamma(s)\le 0$ for $s\in [s_0,1)$ (when this interval is non-empty). We have thus proved the upper bound for $\Gamma$ with $C_{\ref{cst3}}(\varepsilon) \triangleq s_\varepsilon \gamma(1)$. 

Consider next  $s\in [1,\infty) \cap [s_0,\infty)$. Since $\gamma$ is non-increasing,
\begin{equation*}
s \gamma(s) - \gamma(s_0) \le (s-1) \gamma(s) \leq \Gamma(s)\,.
\end{equation*}
If $s\in [s_0,1)$ (when this interval is non-empty), then, using again the monotonicity of $\gamma$,
\begin{equation*}
	s\gamma(s)- \Gamma(s) = s \gamma(s) + \int_s^1 \gamma(\eta)\ \mathrm{d}\eta \le (s+1-s)\gamma(s) \le \gamma(s_0) \,,
\end{equation*}
and the proof is complete.
\end{proof}

We next turn to an upper bound for $\Gamma$ when the growth condition \eqref{gamma2} is satisfied.

\begin{lemma}\label{upG} \refstepcounter{NumConst}\label{cst4}
Under the assumptions \eqref{g1} and \eqref{gamma2}, there is $C_{\ref{cst4}}>0$ depending on $\gamma$ such that, for all $s\geq 1$,
\begin{equation*}
\Gamma(s)\leq \Gamma^*(s)\triangleq\begin{cases}
C_{\ref{cst4}}\log{s},\quad&\text{ when }\;\; l=1\\
 \\
\displaystyle{\frac{C_{\ref{cst4}}(s^{1-l}-1)}{1-l}},\quad&\text{ when }\;\; l\neq 1.
\end{cases}
\end{equation*}
\end{lemma}

\begin{proof}
In view of assumption \eqref{gamma2}, there are $s_0>1$ and $C>0$ such that $s^l\gamma(s)\leq C$ for all $s\geq s_0$. In addition, the monotonicity of $\gamma$ and the non-negativity of $l$ ensure that $s^l\gamma(s)\leq s_0^l\gamma(1)$ for all $1\leq s\leq s_0$. Thus, $s^l\gamma(s)\leq C$ for all $s\geq 1$, from which we deduce that
\begin{equation*}
\Gamma(s)=\int_1^s\gamma(\eta)\mathrm{d}\eta\leq C\int_1^s\eta^{-l}\mathrm{d}\eta=\begin{cases}
C\log{s},\quad&\text{ when }\;\; l=1, \\
 \\
\displaystyle{\frac{C(s^{1-l}-1)}{1-l}},\quad&\text{ when }\;\; l\neq 1,
\end{cases}
\end{equation*}
as claimed.
\end{proof}

We next provide some sufficient conditions on $\gamma$ which guarantee that it satisfies assumption~\eqref{A3}.

\begin{lemma}\label{sufconA3}
Assume that $\gamma\in C^1((0,\infty))$ is positive and one of the following cases holds.
	\begin{itemize}
		\item [(a)]  The function $\gamma$ is non-increasing on $(0,\infty)$ and there is $A>0$ such that
		\begin{equation*}
			\int_s^{2s} \gamma(\eta)\ \mathrm{d}\eta \le A\,, \qquad s\ge 1\,;
		\end{equation*}
		\item [(b)] There are $b_0\in (0,1]$ and $s_1>0$ such that $s\gamma'(s)+b_0\gamma(s)\leq 0$ for all $s\geq s_1$;
		\item [(c)]  There are constants $l>0$ and $0<B_l\leq A_l<\infty$ such that  $(1-l)A_l<B_l$ and
		\begin{equation*}
			B_l=\liminf\limits_{s\rightarrow\infty}s^l\gamma(s)\leq \limsup\limits_{s\rightarrow\infty}s^l\gamma(s)=A_l.
		\end{equation*}
	\end{itemize}
	Then assumption~\eqref{A3} is fulfilled.
\end{lemma}

\begin{proof}
We consider $s_0>0$ and handle the three cases in  different ways.

\noindent Case~(a). In that case, we observe that, for $s\ge \max\{s_0,1\}$, 
\begin{equation*} 
	s \gamma(s) \le 2 \int_{s/2}^s \gamma(\eta)\ \mathrm{d}\eta \le 2A\,,
\end{equation*}
while $s\gamma(s) \le \gamma(s_0)$ for $s\in [s_0,1)$ (when this interval is non-empty). Consequently, assumption~\eqref{A3} is satisfied with $b_0=1$.

\smallskip

\noindent Case~(b). In that case, we may assume $s_1\ge \max\{s_0,1\}$ and find that, for $s\geq s_1$,
\begin{align*}
	s\gamma(s)+(b_0-1)\Gamma(s)=&\int_{s_1}^s\left(\eta\gamma'(\eta)+b_0\gamma(\eta)\right)\rd\eta+s_1\gamma(s_1)+(b_0-1)\Gamma(s_1)\\
	\leq& s_1\gamma(s_1)\,,
\end{align*}
due to the non-negativity of $\Gamma(s_1)$. For $s\in[s_0,s_1)$, we observe that \begin{align*}
	s\gamma(s)+(b_0-1)\Gamma(s)\leq & s_1 \|\gamma\|_{L^\infty(s_0,s_1)} + (1-b_0)\int_{\min\{1,s_0\}}^1\gamma(\eta)\rd\eta \\
	\leq & (1+s_1) \|\gamma\|_{L^\infty(\min\{1,s_0\},s_1)}\,.
\end{align*}
Thus, \eqref{A3} also holds.

\smallskip

\noindent Case~(c). In that case, for any $\varepsilon\in (0,B_l)$,there is $s_{\varepsilon}\ge \max\{1,s_0\}$ such that
	\begin{equation*}
		B_l-\varepsilon\le s^l\gamma(s)\leq A_l+\varepsilon\,, \qquad s\geq s_{\varepsilon}\,.
	\end{equation*}
In particular, for $b_0\in (0,1]$, 
\begin{subequations}\label{y1}
\begin{align}
	s \gamma(s) + (b_0-1) \Gamma(s) & \le  s\gamma(s) + (b_0-1) \int_{s_\varepsilon}^s \gamma(\eta)\ \mathrm{d}\eta \nonumber \\
	& \le (A_l+\varepsilon) s^{1-l} - (1-b_0)(B_l-\varepsilon) \int_{s_\varepsilon}^s \eta^{-l}\ \mathrm{d}\eta\,, \qquad  s\ge s_\varepsilon\,, \label{y1a}
\end{align}
while
\begin{equation}
	s \gamma(s) + (b_0-1) \Gamma(s) \le 2 s_\varepsilon \|\gamma\|_{L^\infty(s_0,s_\varepsilon)}\,, \qquad s\in [s_0,s_\varepsilon]\,. \label{y1b}
\end{equation}
\end{subequations}
At this point, either $l\ge 1$ and we readily deduce from \eqref{y1} with $b_0=1$ and $\varepsilon = \varepsilon_1 = B_l/2$ that \eqref{A3} is satisfied with $b_0=1$ and $K_0(s_0) = A_l + B_l + 2 s_{\varepsilon_1} \|\gamma\|_{L^\infty(s_0,s_{\varepsilon_1})}$.

\noindent Or $l\in (0,1)$. Since $(1-l)A_l<B_l$, we choose 
\begin{equation*}
	b_0 = \frac{B_l-(1-l)A_l}{2B_l}\in (0,1) \;\;\text{ and }\;\; \varepsilon=\varepsilon_2 =	\frac{(1-b_0)B_l-(1-l)A_l}{2-l-b_0} \in (0,B_l)\,.
\end{equation*}
It then follows from \eqref{y1a} that, for any $s\geq s_{\varepsilon_2}$,
\begin{align*}
	s \gamma(s) + (b_0-1) \Gamma(s) & \le  \frac{(1-l)(A_l+\varepsilon_2) - (1-b_0)(B_l-\varepsilon_2)}{1-l} s^{1-l} + \frac{(1-b_0)(B_l-\varepsilon_2)}{1-l} s_{\varepsilon_2}^{1-l} \\
	& \le \frac{(2-l-b_0) \varepsilon_2 - (1-b_0)B_l+ (1-l) A_l}{1-l} s^{1-l} + \frac{B_l}{1-l} s_{\varepsilon_2}^{1-l}  \\
	& = \frac{B_l}{1-l} s_{\varepsilon_2}^{1-l}\,.
\end{align*}
Gathering \eqref{y1b} with the above choice of $b_0$ and the above inequality entails that \eqref{A3} is satisfied in that case as well and completes the proof.
\end{proof}

We finally recall the following lemma given in \cite[Lemma~A.1]{Laur1994}  which we shall use later in Section~\ref{sec3} to complete the Alikakos-Moser iterative argument.

\begin{lemma}\label{lmiter}
	Let $\theta>1$, $b\geq0$, $c\in \mathbb{R}$, $\kappa_0\geq 1$, $\kappa_1\geq 1$, and $\delta_0$ be given numbers such that
	\begin{equation*}
		\delta_0+\frac{c}{\theta-1}>0.
	\end{equation*}
We consider the sequence $(\delta_j)_{j\geq0}$ of real numbers defined by
\begin{equation*}
	\delta_{j+1}=\theta\delta_j+c\,, \qquad j\in\mathbb{N}.
\end{equation*}
Assume further that $(\eta_j)_{j\geq0}$ is a sequence of positive real numbers satisfying 
	\begin{align*}
&\eta_0\leq \kappa_1^{\delta_0},\\
&\eta_{j+1}\leq \kappa_0 \delta_{j+1}^{b} \max\{\kappa_1^{\delta_{j+1}},\eta_j^\theta\}\,, \qquad j\in\mathbb{N}\,.
	\end{align*}	
Then the sequence $(\eta_j^{1/\delta_j})_{j\geq0}$ is bounded.
\end{lemma}

\section{A Two-sided Estimate by a Comparison Argument}\label{sec-comparison}

In this section, we fix initial conditions $(u^{in},v^{in},n^{in})$ satisfying \eqref{ini} and denote the corresponding classical solution to \eqref{cp} given by Theorem~\ref{local} by $(u,v,n)$, which is defined on $[0,T_{\mathrm{max}})$ for some $T_{\mathrm{max}}\in (0,\infty]$. We improve the comparison method proposed in \cite{FuJi2020c} (see also \cite{FuJi2020,LiJi2020}) to develop a two-sided control estimate of $v$ by some auxiliary functions which we define now. Specifically, introducing the operator $\mathcal{A}$ on $L^2(\Omega)$ defined by
 \begin{equation}
	\mathrm{dom}(\mathcal{A}) \triangleq \{ z \in H^2(\Omega)\ :\ \nabla z\cdot \nu = 0 \;\text{ on }\; \partial\Omega\}\,, \qquad  \mathcal{A}z \triangleq - \Delta z + \beta z\,, \quad z\in \mathrm{dom}(\mathcal{A})\,, \label{y2}
\end{equation} 
we recall that $\mathcal{A}$ generates an analytic semigroup on $L^p(\Omega)$ and is invertible on $L^p(\Omega)$ for all $p\in (1,\infty)$. We then set 
\begin{equation}\label{defS}
	S(t) \triangleq \mathcal{A}^{-1}[(u+n)(t)]\ge 0\,,  \qquad t\in [0,T_{\mathrm{max}})\,,
\end{equation} 
and
\begin{equation}\label{defw}
	w(t) \triangleq \mathcal{A}^{-1}[u(t)]\ge 0\,,  \qquad t\in [0,T_{\mathrm{max}})\,,
\end{equation}
the non-negativity of $S$ and $w$ being a consequence of that of $u+n$ and $u$ and the comparison principle. Firstly, due the time continuity of $u$ and $n$, 
\begin{equation*}
	S^{in}\triangleq S(0)=\mA^{-1}[u^{in}+ n^{in}] \;\;\text{ and }\;\; w^{in}\triangleq w(0)=\mA^{-1}[u^{in}]\,,
\end{equation*} 
and it follows from the regularity assumption \eqref{ini} on the initial conditions that $S^{in}$ and $w^{in}$ both belong to $W^{3,N+1}(\Omega)$. 

Secondly, we remark that, due to \cite[Lemma~2.3]{Fuji2016}, \eqref{ini}, \eqref{e0}, and \eqref{uL1}, there are positive constants $S_*$ and $w_*$ depending only on $N$, $\Omega$, $\beta$, and the initial data such that
\begin{equation}\label{lowbsw}
	S\geq S_* \;\;\text{and}\;\; w\geq w_* \;\;\text{ in }\;\; [0,T_{\mathrm{max}})\times\bar{\Omega}\,.
\end{equation}
In addition, since $u\le u+n\le u + \|n^{in}\|_\infty$ in $(0,T_{\mathrm{max}})\times\Omega$ by \eqref{nup} (with $p=\infty$), it readily follows from \eqref{defS}, \eqref{defw}, and the comparison principle that
\begin{equation}
	w \le S \le w + \frac{\|n^{in}\|_\infty}{\beta} \;\;\text{ in }\;\; [0,T_{\mathrm{max}})\times\bar{\Omega}\,. \label{y3}
\end{equation}	

 Now we derive two key identities involving the auxiliary functions $S$ and $w$ and supplement \eqref{lowbsw} with pointwise upper bounds for $S$ and $w$, respectively. 

\begin{lemma}\label{keylem1}
Assume that $\gamma$ satisfies \eqref{g1}. The following two key identities hold
\begin{equation}\label{var0}
	\partial_t S + u\gamma(v) +  n  =\beta\mA^{-1}[ u\gamma(v) + n]
\end{equation} 
and
\begin{equation}\label{var1}
	\partial_t w+u\gamma(v)=\mA^{-1}[\beta u\gamma(v) + uf(n)]
\end{equation}
in $(0,T_{\mathrm{max}})\times \Omega$. Moreover, 
\begin{equation}\label{ptesta}
	w(t,x)\leq	S(t,x)\leq S^{in}(x)e^{\beta\max\{\gamma^*,1\}t}\,, \qquad (t,x)\in [0,T_{\mathrm{max}})\times\bar{\Omega}\,.
\end{equation}
\end{lemma}

\begin{proof} By definition of $\mA$, we rewrite \eqref{cp1} and \eqref{cpn} as:
\begin{equation}\label{recpg1}
	\partial_t u = -\mA[u\gamma(v)] + \beta u\gamma(v) + uf(n) \;\;\text{ in }\;\; (0,T_{\mathrm{max}})\times \Omega
\end{equation}
and
\begin{equation}\label{recpgn}
	\partial_t n = - \mA[n] + \beta n -  uf(n) \;\;\text{ in }\;\; (0,T_{\mathrm{max}}) \times \Omega\,.
\end{equation}
Adding \eqref{recpg1} and \eqref{recpgn}, we obtain
\begin{equation*}
	\partial_t (u + n) + \mA[u\gamma(v) + n] = \beta (u\gamma(v) + n)\;\;\text{ in }\;\; (0,T_{\mathrm{max}}) \times \Omega\,.
\end{equation*}
Applying $\mA^{-1}$ to both sides of the above identity gives the key identity~\eqref{var0}. As for the identity~\eqref{var1}, it simply follows by applying $\mA^{-1}$ on both sides of \eqref{recpg1}.
	
We next infer from \eqref{gvup}, \eqref{var0}, the non-negativity of $\gamma$, $u$, and $n$, and the (elliptic) comparison principle  that,
\begin{align*}
	\partial_t S \le \partial_t S + u\gamma(v) + n & \le \beta \mA^{-1}[u\gamma(v) + n]  \leq \beta \mA^{-1}[\gamma^* u + n] \\ 
	& \leq \beta \max\{\gamma^*,1\} \mA^{-1}[u+ n] = \beta \max\{\gamma^*,1\} S
\end{align*} 
in $(0,T_{\mathrm{max}}) \times \Omega$. Integrating the above differential inequality  with respect to time and using \eqref{y3} give \eqref{ptesta}.
\end{proof}

The next lemma establishes an upper bound on $v$ in terms of $S$ (and thus also in terms of $w$ in view of \eqref{y3}). The proof is the same as \cite[Lemma~7]{FuJi2020c} with a minor modification. 

\begin{lemma}\label{vbd}
	Assume that $\gamma$ satisfies assumption~\eqref{g1} and consider  $\rho\ge 1$ such that $\tau \gamma(\rho)<1$. There is $K_1(\rho)>0$ depending on $\Omega$, $\gamma$, $\tau$, $\beta$, the initial data, and $\rho$ such that
	\begin{equation}\label{vbound}
	v(t,x)\leq \frac{1}{1-\tau\gamma(\rho)}\bigg(S(t,x)+K_1(\rho)\bigg)\,, \qquad (t,x)\in [0,T_{\mathrm{max}}) \times \bar{\Omega}\,.
	\end{equation}
\end{lemma}

\begin{proof}  Introducing the parabolic operator
\begin{equation}
	\ml z \triangleq \tau \partial_t z + \mA z = \tau \partial_t z - \Delta z + \beta z\label{y4}
\end{equation} 
and using \eqref{cp2}, we first observe that
	\begin{align}
	\gamma(v)u & = \gamma(v)(\tau \partial_t v - \Delta v + \beta v) \nonumber \\
	& =\bigg(\tau \partial_t \Gamma(v) - \Delta \Gamma(v) + \beta\Gamma(v) \bigg) + \gamma'(v)|\nabla v|^2+\beta \left(v\gamma(v)-\Gamma(v)\right) \nonumber \\
	& = \ml \Gamma(v) + \gamma'(v) |\nabla v|^2 + \beta \left( v\gamma(v) - \Gamma(v) \right)\,. \label{Gamma}
	\end{align}
	Substituting the  identities \eqref{defS} and \eqref{var0} into  \eqref{cp2} and making use of \eqref{Gamma}, we deduce that
	\begin{align}
	\mathcal{L}v & = \tau \partial_t v - \Delta v + \beta v = u + n - n = \mathcal{L}S - \tau\partial_t S - n \nonumber \\
	& = \mathcal{L}S + \tau (u \gamma(v)+ n) - \beta\tau \mA^{-1}[u\gamma(v) + n] -  n \nonumber \\
	&=\ml \left(S+ \tau \Gamma(v)\right) + \tau \gamma'(v) |\nabla v|^2 + \beta\tau \left( v\gamma(v)-\Gamma(v) \right) -  \beta\tau \mA^{-1}[u\gamma(v) +  n] + (\tau-1)n. \label{compare0}
	\end{align}
Since $\gamma'(v) |\nabla v|^2$ is non-positive by \eqref{g1} and $\mA^{-1}[u\gamma(v) +  n]$ is non-negative by the comparison principle, it follows from \eqref{nup} (with $p=\infty$), \eqref{gvup}, \eqref{compare0}, and Lemma~\ref{lemGam} (with $s_0=v_*$) that
\begin{equation}
	\mathcal{L}v \le \ml \left(S+ \tau \Gamma(v)\right) + \beta\tau \gamma(v_*) + \tau \|n^{in}\|_\infty\,. \label{y5} 
\end{equation}
Now, in view of the assumption \eqref{ini} on the initial data, we may choose a positive constant $K_1\geq \left( \beta\tau \gamma(v_*) + \tau \|n^{in}\|_\infty  \right)/\beta$ such that 
\begin{equation*}
	K_1 \ge v^{in}(x) - S^{in}(x) - \tau \Gamma(v^{in})(x)\,, \qquad x\in\bar{\Omega}\,.
\end{equation*}
We then deduce from \eqref{y5} and the parabolic comparison principle that 
\begin{equation}
	v(t,x) \leq S(t,x) + \tau \Gamma(v(t,x)) + K_1\,, \qquad (t,x) \in [0,T_{\mathrm{max}})\times \bar{\Omega}\,. \label{y6}
	\end{equation}
Finally, pick $\rho\ge 1$ such that $0<\tau\gamma(\rho)<1$, the existence of which is granted by \eqref{g1}. Then, by \eqref{g1},
\begin{equation*}
	\Gamma(s) \le \left\{
	\begin{array}{cl}
		\rho \gamma(1) + (s-\rho) \gamma(\rho) \le \rho \gamma(1) + s \gamma(\rho)\,, & s\in [\rho,\infty)\,, \\
		& \\
		\rho\gamma(1) \le \rho \gamma(1) + s \gamma(\rho)\,, & s\in [v_*,\rho]\,, 
	\end{array}
\right.
\end{equation*}
so that $\tau\Gamma(v)\leq \tau\gamma(\rho)v + \tau\rho\gamma(1)$ in $[0,T_{\mathrm{max}})\times \bar{\Omega}$. Combining this estimate with \eqref{y6} completes the proof.
\end{proof}

Since $\lim\limits_{s\rightarrow\infty}\gamma(s)=0$ by \eqref{g1}, a tight control from $S$ on $v$ follows from Lemma~\ref{vbd} by picking $\rho$ sufficiently large, as reported in the next result.

\begin{proposition}\label{vbd1} Assume \eqref{g1} and consider $\varepsilon>0$. There is $L_1(\varepsilon)>0$ depending on $\Omega$, $\gamma$, $\tau$, $\beta$, the initial data, and $\varepsilon$ such that
\begin{equation}
	v\leq (1+\varepsilon)S + L_1(\varepsilon) \;\;\text{ in }\;\; [0,T_{\mathrm{max}})\times \bar{\Omega}\,. \label{y7a}
\end{equation} 
Moreover, there is $B>0$ depending on $\Omega$, $\gamma$, $\tau$, $\beta$, and the initial data such that, setting $\tilde{w} \triangleq B w$ and $\tilde{S} \triangleq B S$, 
\begin{equation}
v \le \tilde{w} \le \tilde{S}\;\;\text{ in }\;\; [0,T_{\mathrm{max}})\times \bar{\Omega}\,. \label{y7b}
\end{equation}
\end{proposition}

\begin{proof}
By \eqref{g1}, there is $\rho_\varepsilon\ge 1$ such that $1<(1-\tau\gamma(\rho_\varepsilon)) (1+\varepsilon)$. It then follows from  \eqref{vbound} (with $\rho=\rho_\varepsilon$) that
\begin{equation*}
	v \le (1+\varepsilon) S + (1+\varepsilon) K_1(\rho_\varepsilon)\;\;\text{ in }\;\; [0,T_{\mathrm{max}})\times \bar{\Omega}\,,
\end{equation*} 
from which we deduce \eqref{y7a} with $L_1(\varepsilon) \triangleq (1+\varepsilon) K_1(\rho_\varepsilon)$. We next infer from \eqref{lowbsw}, \eqref{y3}, and \eqref{y7a} (with $\varepsilon=1$) that, in $[0,T_{\mathrm{max}})\times \bar{\Omega}$,
\begin{align*}
	v & \le 2 S + \frac{L_1(1)}{S_*} S \le \left( 2 + \frac{L_1(1)}{S_*} \right) \left( w + \frac{\|n^{in}\|_\infty}{\beta} \right) \\ 
	& \le \left( 2 + \frac{L_1(1)}{S_*} \right) \left( w + \frac{\|n^{in}\|_\infty}{\beta w^*} w \right) = B w \le B S 
\end{align*}
with
\begin{equation*}
	B \triangleq \left( 2 + \frac{L_1(1)}{S_*} \right) \left( 1 + \frac{\|n^{in}\|_\infty}{\beta w^*} \right)\,,
\end{equation*}
which completes the proof.
\end{proof}

Next, we aim to derive a reverse relation between $v$ and $S$, which will play a crucial role in Section~\ref{ubcs}. To begin with, we prove the following result.

\begin{lemma}\label{revs0}
Assume \eqref{g1}. There is $K_2>0$ depending on $\Omega$, $\gamma$, $\tau$, $\beta$, and the initial data such that 
\begin{equation*}
S\leq v + \beta\tau \mA^{-1}[\Gamma(v)] + K_2 \;\;\text{ in }\;\; [0,T_{\mathrm{max}})\times \bar{\Omega}\,.
\end{equation*}
\end{lemma}

\begin{proof}
Applying $\mA^{-1}$ to both sides of \eqref{Gamma}, we obtain
\begin{align*}
	\mA^{-1}[u\gamma(v)]& = \mA^{-1}\bigg[ \ml \Gamma(v) +  \gamma'(v) |\nabla v|^2 + \beta \left( v\gamma(v) - \Gamma(v) \right) \bigg]\\
	&=\ml \mA^{-1}\Gamma(v) + \mA^{-1}\bigg[ \gamma'(v)| \nabla v|^2 + \beta \left( v\gamma(v) - \Gamma(v) \right) \bigg]\,.
\end{align*}
We then infer from \eqref{cp2}, \eqref{defS}, \eqref{var0}, and the above identity that
\begin{align*}
	\ml v & = u+n-n = \ml S -\tau \partial_t S - n \\ 
	& = \ml S  + \tau u\gamma(v) + (\tau-1) n - \beta\tau \mA^{-1}[u\gamma(v)+n] \\
	& = \ml S  + \tau u\gamma(v) + (\tau-1) n - \beta\tau \ml\mA^{-1}[\Gamma(v)] \\
	& \qquad - \beta\tau \mA^{-1}\bigg[ n + \gamma'(v) |\nabla v|^2 + \beta (v\gamma(v) - \Gamma(v)) \bigg]\,,
\end{align*}
which implies, together with  the non-negativity of $u\gamma(v)$, $n$, and $-\mA^{-1}[\gamma'(v)|\nabla v|^2]$, that
\begin{equation*}
\ml S\leq \ml \left( v + \beta\tau \mA^{-1}[\Gamma(v)] \right) + \beta\tau \mA^{-1} \big[  n + \beta \left( v\gamma(v) - \Gamma(v) \right) \big] + n\,.
\end{equation*}
We further deduce from \eqref{nup} (with $p=\infty$), \eqref{e00}, \eqref{Gamma0} (with $s_0=v_*$), and the elliptic comparison principle that 
\begin{equation*}
\ml S \leq \ml \left( v + \beta\tau \mA^{-1}[\Gamma(v)] \right) + (1+\tau) \|n^{in}\|_{\infty} + \beta\tau \gamma(v_*)\,.
\end{equation*}
We now take $K_2\ge \left[ (1+\tau) \|n^{in}\|_{\infty} + \beta\tau \gamma(v_*) \right]/\beta$ such that
\begin{equation*}
	K_2 \ge S^{in}(x) -v^{in}(x) - \beta\tau \mA^{-1}[\Gamma(v^{in})](x)\,, \qquad x\in\bar{\Omega}\,,
\end{equation*}
and conclude with the help of the parabolic comparison principle that
\begin{equation*}
S(t,x)\leq v(t,x) + \beta\tau \mA^{-1}[\Gamma(v)](t,x) + K_2\,, \qquad (t,x)\in [0,T_{\mathrm{max}})\times \bar{\Omega}\,.
\end{equation*}
This completes the proof.
\end{proof}

After this preparation, we are in a position to derive an upper bound on $S$ and $w$ in terms of $v$, using additionally either assumption~\eqref{A3} or assumption~\eqref{gamma2} with $l>\frac{(N-4)_+}{N-2}$. We begin with the former. 
	
\begin{proposition}\label{Lemrev}
	Assume \eqref{g1} and \eqref{A3}. There is $A>0$ depending only on $\Omega$, $\gamma$, $\tau$, $\beta$, $b_0$, $K_0(v_*)$, and the initial data such that
\begin{equation*}
	v \ge A S \ge A w \;\;\text{ in }\;\; [0,T_{\mathrm{max}})\times \bar{\Omega}\,.
\end{equation*}
\end{proposition}

\begin{proof}
On the one hand, it follows from \eqref{y7b}, see Proposition~\ref{vbd1}, and the monotonicity of $\Gamma$ that  $\Gamma(v) \le \Gamma(\tilde{S})$ in $[0,T_{\mathrm{max}})\times \bar{\Omega}$ and the elliptic comparison principle ensures that
\begin{equation}
	\mA^{-1}[\Gamma(v)] \le \mA^{-1}[\Gamma(\tilde{S})] \;\;\text{ in }\;\; [0,T_{\mathrm{max}})\times \bar{\Omega}\,. \label{y8}
\end{equation}
On the other hand, since $\tilde{S}=B S$ and $\gamma$, $u$, $n$, and $B$ are all non-negative, we infer from \eqref{g1} and \eqref{defS} that
\begin{align}
	\mA[\Gamma(\tilde{S})] & = - \Delta\Gamma(\tilde{S}) + \beta \Gamma(\tilde{S})  = \gamma(\tilde{S}) \mathcal{A}[\tilde{S}] - \gamma'(\tilde{S}) |\nabla\tilde{S}|^2 + \beta \left( \Gamma(\tilde{S}) - \tilde{S} \gamma(\tilde{S}) \right) \nonumber \\
	& \ge B\gamma(\tilde{S}) (u+n) + \beta \left( \Gamma(\tilde{S}) - \tilde{S} \gamma(\tilde{S}) \right) \ge \beta \left( \Gamma(\tilde{S}) - \tilde{S} \gamma(\tilde{S}) \right) \label{S0}\,.
\end{align}
Furthermore, since $\tilde{S}= B S \ge B S_*$ by \eqref{lowbsw}, assumption~\eqref{A3} (with $s_0=B S_*$) gives
\begin{equation*}
	\Gamma(\tilde{S}) - \tilde{S} \gamma(\tilde{S}) \ge b_0 \Gamma(\tilde{S}) - K_0(B S_*)\,.
\end{equation*}
Hence, recalling \eqref{S0}, 
\begin{equation*}
	\mA[\Gamma(\tilde{S})] \ge \beta b_0 \Gamma(\tilde{S}) - \beta K_0(B S_*) = \mA\left[ \beta b_0 \mA^{-1}[\Gamma(\tilde{S})] - K_0(B S_*) \right]\,,
\end{equation*}
from which we deduce that
\begin{equation}
	\Gamma(\tilde{S}) \ge \beta b_0 \mA^{-1}[\Gamma(\tilde{S})] - K_0(B S_*) \;\;\text{ in }\;\; [0,T_{\mathrm{max}})\times \bar{\Omega}\,, \label{y9}
\end{equation}
by the elliptic comparison principle. It then readily follows from \eqref{y8} and \eqref{y9} that
\begin{equation}
\mA^{-1}[\Gamma(v)]\leq \mA^{-1}[\Gamma(\tilde{S})] \le \frac{\Gamma(\tilde{S})}{\beta b_0} + \frac{K_0(B S_*)}{\beta b_0} \;\;\text{ in }\;\; [0,T_{\mathrm{max}})\times \bar{\Omega}\,. \label{S1}
\end{equation}
Now combining \eqref{S1} and Lemma~\ref{revs0},  we conclude that
\begin{equation*}
\tilde{S}\leq B v+ \frac{\tau B}{b_0} \Gamma(\tilde{S}) + \frac{\tau B K_0(B S_*)}{b_0} + BK_2 \;\;\text{ in }\;\; [0,T_{\mathrm{max}})\times \bar{\Omega}\,.
\end{equation*}
Owing to \eqref{lowbsw}, we may now use Lemma~\ref{lemGam} (with $s_0=B S_*$ and $\varepsilon=b_0/(2\tau B)$) to obtain  
\begin{equation*}
	\tilde{S}\leq B v+\frac{\tilde{S}}{2} + \frac{\tau B}{b_0} C_{\ref{cst3}}\left( \frac{b_0}{2\tau B} \right) + \frac{\tau B K_0(B S_*)}{b_0} + B K_2 \;\;\text{ in }\;\; [0,T_{\mathrm{max}})\times \bar{\Omega}\,.
\end{equation*}
Consequently, introducing
\begin{equation*}
	\frac{1}{A} \triangleq 2 + \frac{1}{v_*} \left[ \frac{2\tau }{b_0} C_{\ref{cst3}}\left( \frac{b_0}{2\tau B} \right) + \frac{2\tau K_0(B S_*)}{b_0}  + 2 K_2 \right]\,, 
\end{equation*}
we deduce from \eqref{e00} and the above inequality that
\begin{equation*}
	\tilde{S}\leq 2 B v + \left[ \frac{2\tau B}{b_0} C_{\ref{cst3}}\left( \frac{b_0}{2\tau B} \right) + \frac{2\tau B K_0(B S_*)}{b_0} + 2 B K_2 \right] \frac{v}{v_*} \le \frac{Bv}{A}\;\;\text{ in }\;\; [0,T_{\mathrm{max}})\times \bar{\Omega}\,.
\end{equation*}
Recalling that $w\le S$ by \eqref{y3} and that $\tilde{S}=BS$, we have thus established that $A w \le A S \le v$ in $[0,T_{\mathrm{max}})\times \bar{\Omega}$ as claimed.
\end{proof}

We finally establish a similar result when $\gamma$ satisfies \eqref{g1} and, either $N\leq3$, or $N\ge 4$ and the parameter $l$ in \eqref{gamma2} satisfies $l>\frac{N-4}{N-2}$. We emphasize that, in the most biologically relevant case $N\leq 3$, the reverse estimate holds without additionally assuming \eqref{gamma2} or \eqref{A3}.
	
\begin{proposition}\label{Lemrevbis}
	Assume \eqref{g1} and that, either $N\leq3$, or $N\ge 4$ and the parameter $l$ in \eqref{gamma2} satisfies $l>\frac{N-4}{N-2}$. There is $A>0$ depending only on $\Omega$, $\gamma$, $\tau$, $\beta$, and the initial data such that
	\begin{equation*}
		v \ge A S \ge A w \;\;\text{ in }\;\; [0,T_{\mathrm{max}})\times \bar{\Omega}\,.
	\end{equation*}
\end{proposition}

\begin{proof}
We begin with the  case $N\leq3$. Owing to the monotonicity and non-negativity \eqref{g1} of $\gamma$, there holds $\Gamma(s) \le \gamma(1) s$ for all $s>0$ and we infer from the elliptic comparison principle that $\mA^{-1}[\Gamma(v)] \le \gamma(1) \mA^{-1}[v]$ in $[0,T_{\mathrm{max}})\times \bar{\Omega}$. Then, 
\begin{equation*}
	\|\mA^{-1}[\Gamma(v)]\|_\infty \le \gamma(1) \|\mA^{-1}[v]\|_\infty \le C \|\mA^{-1} [v]\|_{H^2} \le C \|v\|_2 \le C \|u\|_1 \le C 
\end{equation*}
in $[0,T_{\mathrm{max}})\times \bar{\Omega}$, thanks to \eqref{uL1}, Lemma~\ref{lm2}, and the continuous embedding of $H^2(\Omega)$ in $L^\infty(\Omega)$. Here and throughout the proof, $C$ denotes positive constants depending only on $\Omega$, $\gamma$, $\tau$, $\beta$, and the initial data. Together with \eqref{e00} and Lemma~\ref{revs0}, this estimate implies that 
\begin{equation*}
	S\le v + C + K_2 \le \left( 1 + \frac{C+K_2}{v_*} \right) v \;\;\text{ in }\;\; [0,T_{\mathrm{max}})\times \bar{\Omega}\,,
\end{equation*}
which completes the proof of Proposition~\ref{Lemrevbis} for $N\leq3$.

We next turn to $N\ge 4$ when $\gamma$ additionally satisfies \eqref{gamma2} with $l>\frac{N-4}{N-2}$ and handle in a different way the cases $l>1$, $l=1$, and $l\in \left( \frac{N-4}{N-2} , 1 \right)$.
\begin{itemize}
	\item [-] $l>1$: Since $|\Gamma(v)|\le C$ in $[0,T_{\mathrm{max}})\times \bar{\Omega}$ by \eqref{g1}, \eqref{e00}, and Lemma~\ref{upG}, Proposition~\ref{Lemrevbis} readily follows from the elliptic comparison principle and Lemma~\ref{revs0} in that case.
	\item [-] $l\in \left( \frac{N-4}{N-2} , 1 \right)$: We first note that the lower bound on $l$ guarantees that we can find $q\in (1,N/(N-2))$ satisfying also $q>N(1-l)/2$. Next, by \eqref{g1}, \eqref{e00}, and Lemma~\ref{upG}, the function $\Gamma$ has a sublinear growth at infinity and $|\Gamma(v)|\le C v^{1-l}$ in $[0,T_{\mathrm{max}})\times \bar{\Omega}$. Using again \eqref{uL1}, Lemma~\ref{lm2}, the elliptic comparison principle, and the continuous embedding of $W^{2,q/(1-l)}(\Omega)$ in $L^\infty(\Omega)$, we find
	\begin{align*}
			\|\mA^{-1}[\Gamma(v)]\|_\infty & \le C \|\mA^{-1}\big[ v^{1-l} \big]\|_\infty \le C \|\mA^{-1}\big[ v^{1-l} \big]\|_{W^{2,q/(1-l)}} \\
			& \le C \|v^{1-l}\|_{q/(1-l)} = C \|v\|_{q}^{1-l}\le C \|u\|_1^{1-l} \le C 
	\end{align*}
	in $[0,T_{\mathrm{max}})\times \bar{\Omega}$. We then proceed as above to complete the proof of Proposition~\ref{Lemrevbis} in that case.
	\item [-] $l=1$: We fix $\delta\in (0,2/(N-2))$ and $q\in (1,N/(N-2))$ such that $q>\delta N/2$. Since $|\Gamma(v)|\le C v^\delta$ in $[0,T_{\mathrm{max}})\times \bar{\Omega}$ by \eqref{g1}, \eqref{e00}, and Lemma~\ref{upG}, we argue as in the previous case (with $\delta$ instead of $1-l$) to complete the proof of Proposition~\ref{Lemrevbis}.
\end{itemize}
\end{proof}

\section{Global Existence of Classical Solutions}\label{gecs}

In this section, we consider the existence of global classical solutions under the assumptions of Theorem~\ref{TH1}. We first prove that $v$ is  H\"older continuous with respect to both $t$ and $x$, which allows us to apply the theory developed by Amann in \cite{Aman1990, Aman1995} for non-autonomous linear parabolic equations to derive a $W^{2\theta,p}$-estimate for $w$ with any $p>N$ and $\theta\in((N+p)/2p,1)$. In turn, we will obtain a $W^{1,\infty}$-estimate for $v$, which finally gives rise to $L^p$-estimates for $u$ with any $p>1$ by standard energy estimates.

From now on, we assume that $\gamma$ and the initial data satisfy \eqref{g1} and \eqref{ini}, respectively, and that $(u,v,n)$ is the corresponding solution to \eqref{cp} provided by Theorem~\ref{local} and defined on $[0,T_{\mathrm{max}})$. We fix $t_*\in (0,T_{\mathrm{max}})$ and infer from Theorem~\ref{local} that 
\begin{equation}
	M(t_*) \triangleq \|u(t_*)\|_{C^2(\bar{\Omega})} +  \|v(t_*)\|_{C^2(\bar{\Omega})} +\|n(t_*)\|_{C^2(\bar{\Omega})} < \infty\,. \label{y10}
\end{equation}

Throughout this section, $C$ and $(C_i)_{i\ge 5}$ denote positive constants depending only on $\Omega$, $\gamma$, $f$, $\tau$, $\beta$, the initial data, $t_*$, and $M(t_*)$ introduced in \eqref{y10}. Dependence upon additional parameters will be indicated explicitly.

\subsection{H\"older estimates for $v$}\label{sec.ir1}

Introducing 
\begin{equation}
	\varphi\triangleq\mA^{-1}[u\gamma(v)]\,, \label{defvarphi}
\end{equation}
and
\begin{equation}\label{defpsi}
	\psi\triangleq \mA^{-1}[uf(n)],
\end{equation}
we observe from \eqref{defw} and the key identity \eqref{var1} that $w$ solves the initial boundary value problem
	\begin{subequations}\label{cpv}
		\begin{align}
			\partial_t w + u\gamma(v) & = \beta\varphi+\psi\,, \qquad (t,x)\in (0,T_{\mathrm{max}})\times\Omega\,, \label{cpv1} \\
			- \Delta w + \beta w & = u \,, \qquad (t,x)\in (0,T_{\mathrm{max}})\times\Omega\,, \label{cpv1.5} \\
			\nabla w \cdot \nu & = 0,, \qquad (t,x)\in (0,T_{\mathrm{max}})\times\partial\Omega\,, \label{cpv2} \\
			w(0)  & = w^{in}\,, \qquad x\in\Omega\,. \label{cpv3}
		\end{align}
	\end{subequations}
We now fix $T\in(0,T_{\mathrm{max}})$ and set $J=[0,T]$. By \eqref{e00}, \eqref{lowbsw},  Lemma~\ref{keylem1}, and Lemma~\ref{vbd}, there are positive constants $v^*(T)$ and $w^*(T)$ such that
\begin{align}
	0 < v_* \le v(t,x) & \le v^*(T)\;\;\text{ in }\;\; J\times \bar{\Omega}\,, \label{bv}\\
	0 < w_* \le w(t,x) & \le w^*(T)\;\;\text{ in }\;\; J\times \bar{\Omega}\,, \label{bw}
\end{align}
and we infer from \eqref{gvup}, \eqref{fup}, \eqref{bw}, and the elliptic comparison principle that
\begin{equation}
	0\leq \beta\varphi+\psi = \mA^{-1}[\beta u\gamma(v)+uf(n)]\leq \big( \beta\gamma^* + f^* \big) w \leq \big(\beta \gamma^* + f^* \big) w^*(T) \;\;\text{ in }\;\; J\times\bar{\Omega}\,.\label{phib}
\end{equation}

We now proceed along the lines of \cite[Chapter~V, Section~7]{LSU1968} and \cite[Section~5]{Aman1989} to derive a local energy bound for $v$.
\refstepcounter{NumConst}\label{cst5}

\begin{lemma}\label{lem.impreg1}  Let $\delta\in (0,1)$. There is $C_{\ref{cst5}}(T)>0$ such that, if $\vartheta\in C^\infty(J\times\bar{\Omega})$, $0\le \vartheta\le 1$, $\sigma\in\{-1,1\}$, and $h\in\mathbb{R}$ are such that
	\begin{equation}
		\sigma w(t,x) - h \le \delta\,, \qquad (t,x)\in \mathrm{supp}\,\vartheta\,, \label{smallc1}
	\end{equation}
	then
	\begin{align*}
		& \int_\Omega \vartheta^2 (\sigma w(t) - h)_+^2\ \mathrm{d}x + \frac{\gamma_*(T)}{2} \int_{t_0}^t \int_\Omega \vartheta^2 |\nabla (\sigma w(s) - h)_+|^2\ \mathrm{d}x\mathrm{d}s \\
		& \qquad  \le \int_\Omega \vartheta^2 (\sigma w(t_0) - h)_+^2\ \mathrm{d}x + C_{\ref{cst5}}(T) \int_{t_0}^t \int_\Omega  \left( |\nabla\vartheta|^2 + \vartheta|\partial_t\vartheta| \right) (\sigma w(\tau) - h)_+^2\ \mathrm{d}x\mathrm{d}s \\
		& \qquad\quad + C_{\ref{cst5}}(T) \int_{t_0}^t  \int_{A_{h,\vartheta,\sigma}(s)} \vartheta\ \mathrm{d}x \mathrm{d}s 
	\end{align*}
	for $0\le t_0\le t\le T$, where 
	\begin{equation*}
		A_{h,\vartheta,\sigma}(s)\triangleq \left\{ x\in \Omega\ :\ \sigma w(s,x) > h \right\}\,, \qquad s\in [0,T]\,.
	\end{equation*}
\end{lemma}

\begin{proof}
By \eqref{cpv},
\begin{align}
	\frac{1}{2} \frac{d}{dt} \int_\Omega \vartheta^2 ( \sigma w - h)_+^2\ \mathrm{d}x &  =  \sigma \int_\Omega \vartheta^2 (\sigma w - h)_+ \partial_t w \ \mathrm{d}x + \int_\Omega  (\sigma w - h)_+^2 \vartheta \partial_t\vartheta\ \mathrm{d}x \nonumber \\
	& = - \sigma \int_\Omega \vartheta^2 (\sigma w - h)_+ u \gamma(v) \rd x + \sigma \int_\Omega \vartheta^2 (\sigma w - h)_+ (\beta\varphi+\psi)\ \mathrm{d}x \label{Z1}\\
	& \qquad + \int_\Omega (\sigma w - h)_+^2 \vartheta \partial_t\vartheta \mathrm{d} x\,. \nonumber
\end{align}
Either $\sigma=1$ and it follows from \eqref{g1}, \eqref{defw}, \eqref{cpv1.5}, \eqref{bv}, and the non-negativity of $u$ and $w$ that
\begin{align*}
	- \sigma \int_\Omega \vartheta^2 (\sigma w - h)_+ u \gamma(v) \rd x & \le - \gamma(v^*(T)) \int_\Omega \vartheta^2 (w - h)_+ u\ \mathrm{d}x \\
	& = - \gamma(v^*(T)) \int_\Omega \vartheta^2 (w - h)_+ (\beta w - \Delta w)\ \mathrm{d}x \\
	& \le - \gamma(v^*(T)) \int_\Omega \nabla \left[ \vartheta^2 (w - h)_+ \right]\cdot \nabla w\ \mathrm{d}x \\
	& \le - \gamma(v^*(T)) \int_\Omega \vartheta^2 |\nabla (w - h)_+ |^2\ \mathrm{d}x \\
	& \qquad + 2 \gamma(v^*(T)) \int_\Omega \vartheta |\nabla\vartheta| (w-h)_+ |\nabla (w-h)_+|\ \mathrm{d}x\,.
\end{align*}
Or $\sigma=-1$ and we infer from \eqref{gvup}, \eqref{defw}, \eqref{cpv1.5}, and \eqref{bw} that
\begin{align*}
	- \sigma \int_\Omega \vartheta^2 (\sigma w - h)_+ u \gamma(v) \rd x & \le  \gamma^* \int_\Omega \vartheta^2 (-w - h)_+ u\ \mathrm{d}x \\
	& = \gamma^* \int_\Omega \vartheta^2 (-w - h)_+ (\beta w - \Delta w)\ \mathrm{d}x \\
	& \le \beta\gamma^* w^*(T) \int_\Omega \vartheta^2 (-w - h)_+\ \mathrm{d}x \\	
	& \qquad + \gamma^* \int_\Omega \nabla \left[ \vartheta^2 (-w - h)_+ \right]\cdot \nabla w\ \mathrm{d}x \\
	& \le - \gamma^* \int_\Omega \vartheta^2 |\nabla(-w-h)_+|^2\ \mathrm{d}x \\
	& \qquad + 2 \gamma^* \int_\Omega \vartheta |\nabla\vartheta| (-w-h)_+ |\nabla (-w-h)_+|\ \mathrm{d}x \\
	& \qquad + C(T) \int_\Omega \vartheta^2 (-w - h)_+ \ \mathrm{d}x\,.
\end{align*}
Recalling that $\gamma^*=\gamma(v_*)\ge \gamma(v^*(T))$ by \eqref{g1}, we have thus shown  that
\begin{align*}
	- \sigma \int_\Omega \vartheta^2 (\sigma w - h)_+ u \gamma(v) \rd x & \le - \gamma(v^*(T)) \int_\Omega \vartheta^2 |\nabla(\sigma w-h)_+|^2\ \mathrm{d}x \\
	& \qquad + C(T) \int_\Omega \vartheta |\nabla\vartheta| (\sigma w-h)_+ |\nabla (\sigma w-h)_+|\ \mathrm{d}x \\
	& \qquad + C(T) \int_\Omega \vartheta^2 (\sigma w - h)_+\ \mathrm{d}x\,.
\end{align*}
Inserting this estimate in \eqref{Z1} and using \eqref{phib} and Young's inequality lead us to
\begin{align*}
	\frac{1}{2} \frac{d}{dt} \int_\Omega \vartheta^2 ( \sigma w - h)_+^2\ \mathrm{d}x & \le - \gamma(v^*(T)) \int_\Omega \vartheta^2 |\nabla(\sigma w-h)_+|^2\ \mathrm{d}x \\
	& \qquad + \frac{\gamma(v^*(T))}{2} \int_\Omega \vartheta^2 |\nabla(\sigma w-h)_+|^2\ \mathrm{d}x +  C(T) \int_\Omega |\nabla\vartheta|^2 (\sigma w-h)_+^2\ \mathrm{d}x \\
	& \qquad + C(T) \int_\Omega \vartheta^2 (\sigma w - h)_+\ \mathrm{d}x + \big( \beta \gamma^* + f^* \big) w^*(T) \int_\Omega \vartheta^2 (\sigma w - h)_+\ \mathrm{d}x \\
	& \qquad + \int_\Omega (\sigma w - h)_+^2 \vartheta |\partial_t\vartheta| \mathrm{d} x \\
	& \le -\frac{\gamma(v^*(T))}{2} \int_\Omega \vartheta^2 |\nabla(\sigma w-h)_+|^2\ \mathrm{d}x \\
	& \qquad +  C(T) \int_\Omega \left( |\nabla\vartheta|^2 + \vartheta |\partial_t\vartheta| \right) (\sigma w-h)_+^2\ \mathrm{d}x + C(T) \int_\Omega \vartheta^2 (\sigma w - h)_+\ \mathrm{d}x\,.
\end{align*}
We now use \eqref{smallc1} to estimate from above the last term by
\begin{equation*}
	C(T) \int_\Omega \vartheta^2 (\sigma w - h)_+\ \mathrm{d}x \le \delta \int_{A_{h,\vartheta,\sigma}} \vartheta^2\ \mathrm{d}x \le \int_{A_{h,\vartheta,\sigma}} \vartheta\ \mathrm{d}x\,,
\end{equation*}
and integrate the above differential inequality over $(t_0,t)$ to complete the proof.
\end{proof}

We are now in a position to apply \cite[Chapter~II, Theorem~8.2]{LSU1968} to obtain a H\"older estimate for $w$.

\begin{corollary}\label{regularw}
	 There is $\alpha_T\in (0,1)$ depending on $\Omega$, $\gamma$, $f$, $\tau$, $\beta$, $T$, and the initial data such that $w\in C^{\alpha_T}(J\times\bar{\Omega})$.
\end{corollary}

\begin{proof}
Let $\delta\in (0,1)$. It follows from Lemma~\ref{lem.impreg1} that the estimate \cite[Chapter~II, Equation~(7.5)]{LSU1968} holds true (with parameters $q=r=2(1+\kappa)=(2N+4)/N$ satisfying \cite[Chapter~II, equation~(7.3)]{LSU1968}). Consequently, according to \cite[Chapter~II, Remark~7.2]{LSU1968}, there is $C(T)>0$ such that 
\begin{equation*}
	w\in \hat{\mathcal{B}}_2([0,T]\times\bar{\Omega},w^*(T), C(T),(2N+4)/N,\delta,2/N)\,.
\end{equation*} 
Taking also into account the smoothness of the boundary of $\Omega$ and the H\"older continuity of $w^{in}\in C^{\alpha_0}(\bar{\Omega})$ for some $\alpha_0\in (0,1)$, which stems from the definition $w^{in}=\mA^{-1}[u^{in}]$, the regularity \eqref{ini} of $u^{in}$, and elliptic regularity, we then infer from \cite[Chapter~II, Lemma~8.1 \& Theorem~8.2]{LSU1968} that $w\in C^{\alpha_T}([0,T]\times \bar{\Omega})$ for some $\alpha_T\in (0,1)$  depending on $T$, but also on other parameters as indicated in the statement of Corollary~\ref{regularw}.
\end{proof}

Next, we may use the standard regularity theory of parabolic equations to obtain a H\"older estimate for $v$ as in \cite[Lemma~3.1]{FuSe2021}.

\begin{proposition}\label{prop.impreg2}
The function $v$ belongs to $C^{1+\alpha_T,\alpha_T}(J\times\bar{\Omega})$ with the exponent  $\alpha_T$ given in Corollary~\ref{regularw}.
\end{proposition}

\begin{proof}
Set $r \triangleq \mA^{-1}[v]$ and  $r^{in} \triangleq\mA^{-1}[v^{in}]$. In view of the regularity $v^{in}\in W^{1,N+1}(\Omega)$, there is $\alpha_0\in(0,1/(N+1))$ such that $r^{in}\in C^{2+\alpha_0}(\bar{\Omega})$. Besides, we infer from \eqref{cp2} that $r$ is a solution to
\begin{equation*}
	\begin{split}
		& \tau \partial_t r - \Delta r + \beta r = w\,, \qquad (t,x)\in (0,T_{\mathrm{max}})\times \Omega\,, \\
		& \nabla r\cdot \nu =0\,, \qquad (t,x)\in (0,T_{\mathrm{max}})\times \partial\Omega\,, \\
		& r(0) = r^{in}\,, \qquad x\in \Omega\,,
	\end{split}
\end{equation*}
so that standard regularity theory of heat equations, along with Corollary~\ref{regularw}, ensures that $r$ belongs $C^{1+\alpha_T,2+\alpha_T}(J\times\bar{\Omega})$. As a result, we obtain that $v = \beta r-\Delta r\in C^{1+\alpha_T,\alpha_T}(J\times\bar{\Omega})$.
\end{proof}

\subsection{$L^q$-estimates for $u$}\label{sec.ir2}

Thanks to the just derived H\"older estimates on $v$, we may now proceed as in \cite[Section~6]{Aman1989} to establish bounds on $w$ in appropriate fractional Sobolev spaces, which do not depend on $T_{\mathrm{max}}$.
\refstepcounter{NumConst}\label{cst6}

\begin{lemma}\label{lem.impreg3} Let  $T\in(0,T_{\mathrm{max}})$. For any
	$p\in(1,\infty)$ and $\theta\in \left( \frac{1+p}{2p},1 \right)$, there is $C_{\ref{cst6}}(p,\theta,T)>0$ such that
	\begin{equation*}
	\|w(t)\|_{W^{2\theta,p}} \le C_{\ref{cst6}}(p,\theta,T)\,, \qquad t\in [0,T]\,. 
	\end{equation*}
\end{lemma}

\begin{proof} We set $D\triangleq  (v_*/2,2v^*(T))$ and $J=[0,T]$, where $v_*$ and $v^*(T)$ are defined in \eqref{e00} and \eqref{bv}, respectively. For $s\in D$, we define the elliptic operator $\mathsf{A}(s)$ by $\mathsf{A}(s)z := -\gamma(s) \Delta z$ and the boundary operator $\mathsf{B}z := \nabla z\cdot \nu$. We point out that $\mathsf{A}(s)$ is uniformly elliptic for $s\in D$ since  $\min_D\{\gamma\}>0$ due to \eqref{g1}.  For $t\in J$, let $\tilde{\mathsf{A}}(t)$ be the $L^p$-realization of $(\mathsf{A}(v(t)),\mathsf{B})$ with domain 
	\begin{equation*}
	W_{\mathcal{B}}^{2,p}(\Omega) := \{ z \in W^{2,p}(\Omega)\ :\ \nabla z\cdot \nu = 0 \;\;\text{ on}\;\; \partial\Omega \}\,.
	\end{equation*}
	As in the proof of \cite[Theorem~6.1]{Aman1989}, the assumptions on $\gamma$ and Proposition~\ref{prop.impreg2} guarantee that 
	\begin{equation*}
	\tilde{\mathsf{A}} \in BUC^\alpha(J, \mathcal{L}(W^{2,p}(\Omega),L^p(\Omega)))
	\end{equation*}
	and that $\tilde{\mathsf{A}}(J)$ is a \textit{regularly bounded subset} of $C^\alpha(J,\mathcal{H}(W^{2,p}(\Omega),L^p(\Omega)))$ in the sense of \cite[Section~4]{Aman1988} (or, equivalently, the condition \cite[(II.4.2.1)]{Aman1995} is satisfied). By \cite[Theorem~A.1]{Aman1990}, see also \cite[Theorem~II.4.4.1]{Aman1995}, there is a unique parabolic fundamental solution $\tilde{U}$ associated to $\{ \tilde{\mathsf{A}}(t)\ :\ t\in J \}$ and there exist positive constants $M>0$ and $\omega>0$ such that 
	\begin{equation}
	\|\tilde{U}(t,s)\|_{\mathcal{L}(W^{2,p}(\Omega))} + 	\|\tilde{U}(t,s)\|_{\mathcal{L}(L^p(\Omega))} + (t-s) 	\|\tilde{U}(t,s)\|_{\mathcal{L}(L^p(\Omega),W^{2,p}(\Omega))} \le M e^{\omega(t-s)} \label{irb2}
	\end{equation}
	for $0\le s < t\leq T$. Since $\theta\in \left( \frac{1+p}{2p},1 \right)$, it follows from \cite[Theorem~5.2]{Aman1993} that
	\begin{equation*}
	\left( L^p(\Omega) , W_{\mathcal{B}}^{2,p}(\Omega) \right)_{\theta,p} \doteq W_{\mathcal{B}}^{2\theta,p}(\Omega) \triangleq \{ z \in W^{2\theta,p}(\Omega)\ :\ \nabla z\cdot \nu = 0 \;\;\text{ on}\;\; \partial\Omega \}\,,
	\end{equation*}
	and we infer from \cite[Lemma~II.5.1.3]{Aman1995} that there is $M_\theta>0$ such that
	\begin{equation}
	\|\tilde{U}(t,s)\|_{\mathcal{L}(W_{\mathcal{B}}^{2\theta,p}(\Omega))} + (t-s)^\theta \|\tilde{U}(t,s)\|_{\mathcal{L}(L^p(\Omega),W_{\mathcal{B}}^{2\theta,p}(\Omega))} \le M_\theta e^{\omega(t-s)} \label{irb3}
	\end{equation}
	for $0\le s<t\leq T$. We deduce from \eqref{cpv} that $w$ solves 
	\begin{equation}
	\begin{split}
	\partial_t w + \tilde{\mathsf{A}}(\cdot) w  & = F\,, \qquad t\in J\,, \\
	w(0) & = w^{in} \,, 
	\end{split}\label{irb4}
	\end{equation}
	where
	\begin{equation*}
	F(t,x) \triangleq  \big( \beta\varphi + \psi - \beta w\gamma(v)\big)(t,x)\,, \qquad (t,x)\in J\times \Omega\,. 
	\end{equation*}
	We recall that,  according to \eqref{gvup}, \eqref{bw}, and \eqref{phib}, \refstepcounter{NumConst}\label{cst7}
	\begin{equation}
	\| F(t)\|_\infty \le C_{\ref{cst7}}(T) \triangleq \big( \beta \gamma^* + f^* \big) w^*(T) + \beta w^*(T) \gamma^*\,, \qquad t\in J\,, \label{irb5}
	\end{equation}
 	while the continuity of $u$, $v$ and $w$, see Theorem~\ref{local}, ensures that
	\begin{equation}
	F \in C(J\times\bar{\Omega})\,. \label{irb8}
	\end{equation}
	
	Owing to  \eqref{cpv} and \eqref{irb8}, we observe that $w$ is the unique  classical solution to \eqref{irb4} on $J$ and thus has the representation formula 
	\begin{equation}
	w(t) =  \tilde{U}(t,0) w^{in} + \int_0^t \tilde{U}(t,s) F(s)\ \mathrm{d}s\,, \qquad t\in J\,,\label{irb9}
	\end{equation}
	 according to \cite[Remarks~II.2.1.2~(a)]{Aman1995}.  We then infer from  \eqref{irb3}, \eqref{irb5}, and \eqref{irb9} that, for $t\in J $,
	\begin{align}
	\|w(t)\|_{W^{2\theta,p}} & \le M_\theta e^{\omega t} \|w^{in}\|_{W^{2\theta,p}} + M_\theta \int_0^t (t-s)^{-\theta} e^{\omega(t-s)} \|F(s)\|_p\ \mathrm{d}s \nonumber \\
	& \le M_\theta e^{\omega T} \|w^{in}\|_{W^{2,p}} + C_{\ref{cst7}}(T) M_\theta|\Omega|^{1/p} \int_0^t (t-s)^{-\theta} e^{\omega(t-s)} \ \mathrm{d}s\,. \label{irb11}
	\end{align}
Since
	\begin{equation*}
\int_0^t (t-s)^{-\theta} e^{\omega(t-s)} \ \mathrm{d}s=	\int_0^t s^{-\theta} e^{\omega s}\ \mathrm{d}s \leq \frac{T^{1-\theta}e^{\omega T}}{1-\theta}\,,  \qquad t\in J\,,
	\end{equation*}
we conclude from \eqref{irb11} that
	\begin{equation*}
	\|w(t)\|_{W^{2\theta,p}} \le C(p,\theta,T)\,, \qquad t\in J\,,
	\end{equation*}
and the proof is complete.
\end{proof}

With an appropriate choice of $\theta$ and $p$ in Lemma~\ref{lem.impreg3}, we are able to prove that $v$ is bounded in $W^{1,\infty}(\Omega)$ for positive times, which will further give rise to $L^q$-estimates for $u$ with any $q\in (1,\infty)$.

\begin{proposition}\label{propLpu} \refstepcounter{NumConst}\label{cst8}
	Let $T\in (t_*, T_{\mathrm{max}})$. For any  $q\in (1,\infty)$,  there is $C_{\ref{cst8}}(q,T)>0$ such that
	\begin{equation*}
	\|u(t)\|_{q} \le C_{\ref{cst8}}(q,T)\,, \qquad t\in [0,T]\,. 
	\end{equation*}
\end{proposition}

\begin{proof} We fix $p\in (N,\infty)$ and 
\begin{equation*}
	\frac{N+p}{2p} < \theta' < \theta <1\,.
\end{equation*} 
From \eqref{defw} and Lemma~\ref{lem.impreg3}, we deduce that
\begin{equation}
	\|u(t)\|_{W^{2\theta-2,p}} = \|\mA[w(t)]\|_{W^{2\theta-2,p}} \leq C \|w(t)\|_{W^{2\theta,p}} \leq C(T)\,, \qquad t\in [0,T]\,. \label{y11}
\end{equation}
Introducing
\begin{equation*}
	W_{\mathcal{B}}^{\xi,p}(\Omega) \triangleq \left\{ 
	\begin{array}{ll}
		\{ z \in W^{\xi,p}(\Omega)\ :\ \nabla z\cdot \nu = 0 \;\;\text{ on}\;\; \partial\Omega \}\,, &  (p+1)/p < \xi \le 2\,, \\
		& \\
		W^{\xi,p}(\Omega)\,, & (1-p)/p < \xi < (p+1)/p \,,
		\end{array}
	\right.
\end{equation*}
see \cite[Section~7]{Aman1993}, and observing that the choice of $\theta$ and $p$ guarantees that $2\theta-2>(1-p)/p$, the realization in $W_{\mathcal{B}}^{2\theta-2,p}(\Omega)$ of $\mA$ generates an analytic semigroup in $W_{\mathcal{B}}^{2\theta-2,p}(\Omega)$,  see \cite[Theorem~8.5]{Aman1993}, and the interpolation space is characterized by
\begin{equation*}
	\left( W_{\mathcal{B}}^{2\theta-2,p}(\Omega) , W_{\mathcal{B}}^{2\theta,p}(\Omega)\right)_{1+\theta'-\theta,p} = W_{\mathcal{B}}^{2\theta',p}(\Omega)
\end{equation*}
(up to equivalent norms). We then infer from \eqref{cp2}, \eqref{y10}, \eqref{y11}, and \cite[Theorem~V.2.1.3]{Aman1995} that, for $t\in [t_*,T]$, 
\begin{align*}
	\|v(t)\|_{W^{2\theta',p}} & \le \left\| e^{-(t-t_*)\mA} v(t_*) \right\|_{W^{2\theta',p}} + \int_{t_*}^t \left\| e^{-(t-s)\mA} u(s) \right\|_{W^{2\theta',p}}\ \mathrm{d}s \\
	& \le C e^{-\beta t/2} \|v(t_*)\|_{W^{2\theta',p}} + C \int_{t_*}^t (t-s)^{\theta-\theta'-1} e^{-\beta(t-s)/2}\|u(s)\|_{W^{2\theta-2,p}}\ \mathrm{d}s \\
	& \le C \|v(t_*)\|_{W^{2,p}} + C(T) \int_{t_*}^t (t-s)^{\theta-\theta'-1} e^{-\beta(t-s)/2}\ \mathrm{d}s \\
	& \le C(T) (1+ M(t_*))\,.
\end{align*}
As the choice of $\theta'$ guarantees that $W^{2\theta',p}(\Omega)$ is continuously embedded in $C^1(\bar{\Omega})$, the above estimate implies that 
\begin{equation}
	\|\nabla v(t)\|_\infty \le C(T)\,, \qquad t\in [t_*,T]\,. \label{irb14}
\end{equation}

Having at hand the $W^{1,\infty}$-boundedness of $v$ in $[t_*,T]$, it becomes simple to derive the $L^q$-boundedness of $u$ with any $q>1$. To see this, let us multiply \eqref{cp1} by $u^{q-1}$ and integrate by parts to obtain that
\begin{equation}
	\begin{split}
	\frac{1}{q}\frac{d}{dt}\int_\Omega u^q\ \rd x & +(q-1)\int_\Omega \gamma(v)u^{q-2}|\nabla u|^2\ \rd x \\
	& =-(q-1)\int_\Omega \gamma'(v)u^{q-1}\nabla u\cdot\nabla v\ \rd x+\int_\Omega u^qf(n)\ \rd x\,. 
	\end{split}\label{uLq0}
\end{equation}
By Young's inequality, the boundedness \eqref{fup} of $f(n)$,  and the boundedness of $\frac{|\gamma'(v)|^2}{\gamma(v)}$ on $[0,T]\times \bar{\Omega}$ due to \eqref{g1}, \eqref{e00}, and \eqref{bv}, we infer that, for $t\in[t_*,T]$,
\begin{align*}
	&-(q-1)\int_\Omega \gamma'(v)u^{q-1}\nabla u\cdot\nabla v\,\rd x+\int_\Omega u^qf(n)\,\rd x\\
	&\leq \frac{q-1}{2}\int_\Omega \gamma(v)u^{q-2}|\nabla u|^2\,\rd x+\frac{q-1}{2}\int_\Omega \frac{|\gamma'(v)|^2}{\gamma(v)}u^q|\nabla v|^2\ \rd x+f^*\int_\Omega u^q\ \rd x\\
	&\leq \frac{q-1}{2}\int_\Omega \gamma(v)u^{q-2}|\nabla u|^2\,\rd x + \big( C(q,T) \sup\limits_{[t_*,T]}\|\nabla v\|_{\infty}+f^* \big)\int_\Omega u^q\,\rd x\,.
\end{align*}
We then use \eqref{irb14} to deduce that, for $t\in[t_*,T]$,
\begin{equation*}
	\frac{d}{dt}\int_\Omega u^q \,\rd x+\frac{q(q-1)}{2}\int_\Omega \gamma(v)u^{q-2}|\nabla u|^2\,\rd x\leq C(q,T) \int_\Omega u^q\,\rd x\,. 
\end{equation*}
Integrating the above differential  inequality over $[t_*,t]$ for $t\in [t_*,T]$ provides the boundedness of $\|u\|_q$ in $[t_*,T]$, while that on $[0,t_*]$ is a consequence of Theorem~\ref{local}.
 \end{proof}

\begin{proof}[Proof of Theorem~\ref{TH1}]
	With the aid of Proposition~\ref{propLpu}, we may further use a standard bootstrap argument to prove that, for any $0<T<T_{\mathrm{max}}$, there is $C(T)>0$ independent of $T_{\mathrm{max}}$ such that
\begin{equation*}
\sup\limits_{0\leq t\leq T}\|u(t)\|_{\infty}\leq C(T)\,.
\end{equation*}
According to Theorem~\ref{local}, we deduce that $T_{\mathrm{max}}= \infty$ and thus Theorem~\ref{TH1} is proved.
\end{proof}

\section{Uniform-in-time Boundedness in $2$D}\label{sec-2D}

 The main outcome of Theorem~\ref{TH1} being the global existence of classical solutions to \eqref{cp} under the sole assumption \eqref{g1}, we are next interested in the uniform-in-time boundedness of these solutions under certain decay assumptions as stated in Theorem~\ref{TH0} and Theorem~\ref{TH2}.
  
In this section, we first show the uniform-in-time boundedness of classical solutions to \eqref{cp} in the two-dimensional case. The proof is based on the particular smoothing properties of elliptic equations with right-hand side in $L^1(\Omega)$, as exemplified in Lemma~\ref{lm2e}. Similar arguments are carried out  in \cite[Section~3.1]{FuJi2020b} for the two-component system~\eqref{prlmod0}. 

Throughout this section, $C$ and $(C_i)_{i\ge 9}$ denote positive constants depending only on $\Omega$, $\gamma$, $f$, $\tau$, $\beta$, and the initial data. Dependence upon additional parameters will be indicated explicitly.

\medskip

To begin with, we derive a time-independent $H^1$-bound on $w$. \refstepcounter{NumConst}\label{cst9}

\begin{lemma}\label{wH1a} Assume \eqref{g1} and \eqref{A2b} and that $\|u^{in}+n^{in}\|_1<4\pi/\chi$ for some $\chi>0$. There is $C_{\ref{cst9}}(\chi)>0$ such that
	\begin{equation*}
	\sup\limits_{t\geq0}\left(\|\nabla w(t)\|_2+\|w(t)\|_2+\int_t^{t+1}\int_\Omega u^2\gamma(v)\,\rd x\rd s\right)\leq C_{\ref{cst9}}(\chi)\,.
	\end{equation*}
\end{lemma}

\begin{proof}
We multiply the key identity \eqref{var1} by $u$ and integrate over $\Omega$. Recalling that $w=\mA^{-1}[u]$ by \eqref{defw} and that $\mA^{-1}$ is self-adjoint, we obtain 
\begin{align*}
	\frac{1}{2}\frac{d}{dt}\left(\|\nabla w\|_2^2 + \beta \|w\|_2^2\right)+\int_\Omega u^2\gamma(v)\,\rd x=&\int_\Omega u \mA^{-1}[\beta u\gamma(v)+uf(n)]\,\rd x\\
	=&\int_\Omega (\beta u\gamma(v)+uf(n)) \mA^{-1}[u]\,\rd x\\
	=&\int_\Omega (\beta \gamma(v)+f(n)) uw\, \rd x\,.
\end{align*}
Hence, owing to \eqref{gvup} and \eqref{fup},
\begin{equation*}
	\frac{1}{2}\frac{d}{dt}\left(\|\nabla w\|_2^2 + \beta \|w\|_2^2\right)+\int_\Omega u^2\gamma(v)\,\rd x \le (\beta \gamma^* + f^*) \int_\Omega uw\, \rd x\,.
\end{equation*}
Also, it follows from \eqref{defw} that
\begin{equation*}
	\|\nabla w\|_2^2+\beta\|w\|_2^2=\int_\Omega w u\, \rd x.
\end{equation*} 
Combining the above inequalities and using Young's inequality, we arrive at
\begin{align*}
	\frac{d}{dt}\left(\|\nabla w\|_2^2 + \beta\|w\|_2^2\right) & +\|\nabla w\|_2^2 + \beta \|w\|_2^2 +2\int_\Omega u^2 \gamma(v)\,\rd x\\
	=&(2\beta\gamma^* + 2f^* + 1) \int_\Omega w u \,\rd x\\
	\leq&\int_\Omega u^2 \gamma(v)\, \rd x
	+\frac{(2\beta\gamma^*+2f^*+1)^2}{4}\int_\Omega \frac{w^2}{\gamma(v)}\,\rd x.
\end{align*} 
We thus  obtain
\begin{equation}\label{wb00}
	\frac{d}{dt}\left(\|\nabla w\|_2^2 + \beta \|w\|_2^2\right) + \|\nabla w\|_2^2 + \beta \|w\|_2^2	+\int_\Omega u^2 \gamma(v)\, \rd x 
	\leq C\int_\Omega \frac{w^2}{\gamma(v)}\, \rd x\,.
\end{equation} 

We now deduce from the assumption~\eqref{A2b} that there exist $b>0$  and $s_\chi>v_*$ depending on $\chi$ such that  $e^{\chi s} \gamma(s) \ge 1/b$ for all $s\geq s_\chi$. Since $\gamma(s)\ge \gamma(s_\chi)$ for $s\in (0,s_\chi]$, we end up with
\begin{equation}\label{cond_gamma0}
	\frac{1}{\gamma(s)}\leq \max\left\{ b , \frac{1}{\gamma(s_\chi)} \right\} e^{\chi s}\,, \qquad s>0\,.
\end{equation}
Now, for $\varepsilon>0$, we infer from \eqref{y3}, \eqref{y7a}, \eqref{cond_gamma0}, and the elementary inequality $e^{\varepsilon s} \ge \varepsilon^2 s^2$, $s>0$,  that
\begin{equation}
\begin{split}
	\int_\Omega \frac{w^2}{\gamma(v)}\, \rd x & \le C(\chi) \int_\Omega w^2 e^{\chi v}\, \rd x \le C(\chi) \int_\Omega S^2 e^{\chi(1+\varepsilon)S + \chi L_1(\varepsilon)}\, \rd x \\
	& \le \frac{C(\chi) e^{\chi L_1(\varepsilon)}}{\chi^2 \varepsilon^2} \int_\Omega e^{\chi(1+2\varepsilon)S}\, \rd x \,.
\end{split} \label{unic0}
\end{equation}
Since $\|u+n\|_1=\|u^{in}+n^{in}\|_1 < 4\pi/\chi$ by \eqref{massini} and \eqref{e0}, we may choose $\varepsilon_\chi>0$ such that
\begin{equation}
	\chi (1+ 2 \varepsilon_\chi) < \frac{4\pi}{\|u^{in}+n^{in}\|_1} \qquad \left( \text{ say }\;\;  \varepsilon_\chi \triangleq \frac{\pi}{\chi \|u^{in}+n^{in}\|_1} - \frac{1}{4} \right) \label{y13}
\end{equation}
and deduce from \eqref{defS} and Lemma~\ref{lm2e} that
\begin{equation}
	\int_\Omega e^{\chi(1+2\varepsilon_\chi)S}\, \rd x \le C_{\ref{cst2}}(\|u^{in}+n^{in}\|_1,\chi(1+2\varepsilon_\chi))\,. \label{y12}
\end{equation}
Gathering \eqref{wb00}, \eqref{unic0} (with $\varepsilon=\varepsilon_\chi$), and \eqref{y12} gives
\begin{equation*}
	\frac{d}{dt}\left(\|\nabla w\|_2^2 + \beta \|w\|_2^2\right) + \|\nabla w\|_2^2 + \beta \|w\|_2^2	+\int_\Omega u^2 \gamma(v)\, \rd x 
	\leq C(\chi)\,,
\end{equation*} 
from which Lemma~\ref{wH1a} follows after integration with respect to time.
\end{proof}

\begin{proposition} \refstepcounter{NumConst}\label{cst10}
Assume \eqref{g1} and \eqref{A2b} and that $\|u^{in}+n^{in}\|_1<4\pi/\chi$ for some $\chi>0$.	There is $C_{\ref{cst10}}(\chi)>0$ such that
\begin{equation}
	\sup\limits_{t\geq0}\left(\|w(t)\|_{\infty}+\|v(t)\|_{\infty}\right)\leq C_{\ref{cst10}}(\chi)\,.
\end{equation}
\end{proposition}

\begin{proof} Let $p\in (1,2)$ and $\varepsilon>0$ to be specified later. Since $u=\mA[w]$ by \eqref{defw}, we infer from the Sobolev embedding theorem, H\"{o}lder's inequality,  and \eqref{cond_gamma0} that
\begin{align}\label{winfty}
	\|w\|_{\infty} \leq &C(p) \|w\|_{W^{2,p}} \le C(p)\|u\|_{p} \nonumber\\
	\leq & C(p)\left(\int_\Omega u^2\gamma(v)\,\rd x\right)^{\frac{1}{2}} \left(\int_\Omega (\gamma(v))^{-\frac{p}{2-p}}\,\rd x\right)^{\frac{2-p}{2p}}\nonumber\\
	\leq&C(p,\chi) \left(\int_\Omega u^2\gamma(v)\,\rd x\right)^{\frac{1}{2}} \left(\int_\Omega e^{\frac{\chi p}{2-p}v}\,\rd x\right)^{\frac{2-p}{2p}} \nonumber\\
	\leq& C(p,\chi,\varepsilon) \left(\int_\Omega u^2\gamma(v)\,\rd x\right)^{\frac{1}{2}} \left(\int_\Omega e^{\frac{\chi p}{2-p}(1+\varepsilon)S}\,\rd x\right)^{\frac{2-p}{2p}}\,.
\end{align}
Choosing $p=p_\chi \triangleq \frac{2+4\varepsilon_\chi}{2+3\varepsilon_\chi}\in (1,2)$ with $\varepsilon_\chi$ defined in \eqref{y13}, we observe that
\begin{equation*}
	\frac{\chi p_\chi}{(2-p_\chi)} (1+\varepsilon_\chi) = \chi (1+2\varepsilon_\chi)
\end{equation*}
and we deduce from \eqref{y12} and \eqref{winfty} (with $\varepsilon=\varepsilon_\chi$ and $p=p_\chi$) that
\begin{equation*}
	\|w\|_{\infty}\leq C(\chi) \left(\int_\Omega u^2\gamma(v)\,\rd x\right)^{\frac{1}{2}}.
\end{equation*}	
Combining the above estimate with Lemma~\ref{wH1a}, we obtain, for $t\geq 0$,
\begin{equation}
	\int_t^{t+1}\|w(s)\|_{\infty}\,\rd s\leq \left( \int_t^{t+1} \|w(s)\|_\infty^2\,\rd s \right)^{\frac{1}{2}} \le C(\chi) \left( \int_t^{t+1}\int_\Omega u^2\gamma(v)\,\rd x\rd s \right)^{\frac{1}{2}}\leq C(\chi)\,. \label{y14}
\end{equation}
Now, observing that \eqref{gvup}, \eqref{fup}, \eqref{var1}, the non-negativity of $u$ and $\gamma$, and the elliptic comparison principle imply that
\begin{equation*}
	\partial_t w \le \partial_t w + u\gamma(v) = \mA^{-1}\left[ \beta u\gamma(v) + u f(n) \right] \le (\beta \gamma^* + f^*) \mA^{-1}[u] = (\beta \gamma^* + f^*) w\,,
\end{equation*}
we realize that $w(t+1,x) \le e^{(\beta \gamma^* + f^*)(t+1-s)} w(s,x)$ for all $(s,x)\in (t,t+1)\times\Omega$. Consequently,
\begin{equation*}
	\|w(t+1)\|_\infty e^{(\beta \gamma^* + f^*)(s-t-1)} \le \|w(s)\|_\infty\,, \qquad s\in [t,t+1]\,,
\end{equation*}
and it follows from \eqref{y14} and the above inequality after integration with respect to $s$ over $[t,t+1]$ that 
\begin{equation*}
	\|w(t+1)\|_\infty \frac{1 - e^{-(\beta \gamma^* + f^*)}}{\beta\gamma^*+f^*} \le C(\chi)\,.
\end{equation*}
Since we already know that $\|w(t)\|_\infty \le C$ for $t\in [0,1]$ by Lemma~\ref{keylem1}, we have thus proved that
\begin{equation*}
	\|w(t)\|_\infty \leq C(\chi) \qquad\text{for}\;\;t\geq 0\,.
\end{equation*} 
The claimed boundedness of $v$ now readily follows in view of \eqref{y7b} and the non-negativity of $v$.
\end{proof}

\begin{proof}[Proof of Theorem~\ref{TH0}]
	With Proposition~\ref{propLpu} at hand, we argue in the same manner as in the proof of Theorem~\ref{TH2} in the next section, to which we refer. 
\end{proof}

\section{Uniform-in-time Boundedness in Higher Dimensions}\label{ubcs}

In this section, we turn to the higher dimensional case $N\geq3$ and recall that, besides \eqref{g1}, we assume that the growth of $1/\gamma$ at infinity is monitored by \eqref{gamma2} with $k<N/(N-2)$ and $k-l<2/(N-2)$. We further assume that, either the growth of $s\mapsto s\gamma(s)$ is balanced by that of $\Gamma$ as requested by \eqref{A3}, or that the parameter $l$ in \eqref{gamma2} satisfies $l>(N-4)_+/(N-2)$.

Throughout this section, $C$ and $(C_i)_{i\ge 11}$ denote positive constants depending only on $\Omega$, $\gamma$, $f$, $\tau$, $\beta$, the initial data, the constants $B$ and $A$ defined in Proposition~\ref{vbd1} and in Propositions~\ref{Lemrev} and~\ref{Lemrevbis}, respectively, and the parameters $k$ and $l$ involved in \eqref{gamma2}. Dependence upon additional parameters will be indicated explicitly.

\subsection{Time-independent upper bound estimates}\label{sec3}

First, we derive a uniform-in-time upper bound  for $v$  by a Moser type iteration under the assumptions of Theorem~\ref{TH2}. More precisely, we prove the following result.

\begin{proposition}\label{propunib}
 There is $v^*>0$ depending only on $\Omega$, $\gamma$, $f$, $\tau$, $\beta$, the initial data, the constants $B$ and $A$ defined in Proposition~\ref{vbd1} and in  Propositions~\ref{Lemrev} and~\ref{Lemrevbis}, respectively, and the parameters $k$ and $l$ involved in \eqref{gamma2} such that
\begin{equation*}
	\sup\limits_{t\geq0} \|v(t)\|_{\infty}\leq v^*\,.
\end{equation*}
\end{proposition}

 The proof of Proposition~\ref{propunib} consists of the following lemmas. We recall that, by Proposition~\ref{vbd1} and Propositions~\ref{Lemrev} and~\ref{Lemrevbis},
 \begin{equation}
 	\hat{S}\triangleq A S \leq v \leq \tilde{S}=B S \;\;\text{ in }\;\; [ 0,\infty)\times\bar{\Omega}\,. \label{y20}
 \end{equation} 

\begin{lemma}\label{lem01}
Introducing the function
\begin{equation}
	\gamma_0(s) \triangleq \min\left\{ 1 , \frac{1}{\gamma^*} \right\} \gamma(s)\,, \qquad s>0\,, \label{y21}
\end{equation} 
the function $\tilde{S}$ satisfies the following differential inequality \refstepcounter{NumConst}\label{cst11}
\begin{equation}
\partial_t \tilde{S} +  B\gamma_0(\tilde{S})(u+n) \le \frac{\beta B}{A} \Gamma(\tilde{S}) + C_{\ref{cst11}} \;\;\text{ in }\;\; (0,\infty)\times \Omega\,. \label{e1a}
\end{equation}
\end{lemma}

\begin{proof}
First, we recall the key identity \eqref{var0}:
\begin{equation*}
	\partial_t S + u\gamma(v) + n = \beta \mA^{-1}[u\gamma(v)+n]\;\;\text{ in }\;\; (0,\infty)\times \Omega\,.
\end{equation*}
On the one hand, thanks to \eqref{gvup}, \eqref{y20}, and the non-increasing property of $\gamma_0$, there holds
\begin{equation*}
	B(u\gamma(v)+n)\geq B \left( u \gamma(v) + \frac{n}{\gamma^*} \gamma(v) \right) \ge B \gamma_0(v) (u+n) \geq B \gamma_0(\tilde{S})(u+n)\,.
\end{equation*}
It follows that
\begin{equation}\label{var0B}
	\partial_t\tilde{S}+ B \gamma_0(\tilde{S})(u+n)\leq \beta B \mA^{-1}[u\gamma(v)+n]\;\;\text{ in }\;\; (0,\infty)\times \Omega\,.
\end{equation}

On the other hand, recalling that $\varphi=\mA^{-1}[u\gamma(v)]$, see \eqref{defvarphi}, we infer from  \eqref{y20}, the definition \eqref{defS} of $S$, and the monotonicity \eqref{g1} of $\gamma$ that
\begin{align*}
	A \mathcal{A}[\varphi] &=  A u\gamma(v) \leq A(u+n)\gamma(v) \leq A(u+n)\gamma(\hat{S}) = A \gamma(\hat{S})\mathcal{A}[S] = (\beta \hat{S}-\Delta\hat{S})\gamma(\hat{S})\\
	&=-\Delta \Gamma(\hat{S}) + \beta \Gamma(\hat{S}) + \gamma'(\hat{S}) |\nabla \hat{S}|^2 + \beta \left( \hat{S}\gamma(\hat{S}) - \Gamma(\hat{S}) \right) \\
	& \le \mA[\Gamma(\hat{S})] + \beta \left( \hat{S}\gamma(\hat{S}) - \Gamma(\hat{S}) \right) \;\;\text{ in }\;\; (0,\infty)\times \Omega\,.
\end{align*}
In view of \eqref{lowbsw} and Lemma~\ref{lemGam} (with $s_0=AS_*$), we obtain 
\begin{equation*}
A \mathcal{A}[\varphi]\leq \mathcal{A}[\Gamma(\hat{S})] + \beta  \gamma(A S_*) = \mathcal{A}[\Gamma(\hat{S}) + \gamma(A S_*)]\;\;\text{ in }\;\; (0,\infty)\times \Omega\,,
\end{equation*}
while \eqref{cp3} guarantees that $\nabla \varphi\cdot \nu = \nabla (\Gamma(v) + \gamma(A S_*))\cdot \nu=0$ on $(0,\infty)\times\partial\Omega$, from which we deduce, according to \eqref{y20}, the elliptic comparison principle, and the increasing property of $\Gamma$,  that
\begin{equation*}
	A \varphi\leq \Gamma(\hat{S}) + \gamma(A S_*) \leq \Gamma(\tilde{S}) + \gamma(A S_*) \;\;\text{ in }\;\; (0,\infty)\times \Omega\,.
\end{equation*}
In addition, using again the elliptic comparison principle, along with \eqref{nup} (with $p=\infty$), gives
\begin{align*}
	\|\mA^{-1}[n(t)]\|_{\infty}\leq \frac{\|n(t)\|_{\infty}}{\beta} \leq \frac{\|n^{in}\|_{\infty}}{\beta} \;\;\text{ for }\;\; t\ge 0\,.
\end{align*}
Combining the above inequalities with  \eqref{var0B} completes the proof.
\end{proof}

As in \cite{JiLau2021a}, we next combine \eqref{gamma2} and \eqref{e1a} to derive a differential inequality for $\|\tilde{S}\|_p$ for $p\in (1+k,\infty)$.

\refstepcounter{NumConst}\label{cst12} 
\refstepcounter{NumConst}\label{cst13}
 
\begin{lemma}\label{lmp} 
There are positive constants $C_{\ref{cst12}}>0$ and $C_{\ref{cst13}}>0$ such that, for any $p>1+k$,
\begin{equation*}
	\frac{d}{dt}\|\tilde{S}\|_p^{p} + C_{\ref{cst12}}\frac{p(p-k-1)}{(p-k)^2} \|\nabla \tilde{S}^{\frac{p-k}{2}}\|_2^2 + C_{\ref{cst12}} p\| \tilde{S}\|_{p-k}^{p-k}\leq  C_{\ref{cst13}} p\int_\Omega\left( \tilde{S}^{p-1}\Gamma(\tilde{S})+ \tilde{S}^{p-1}\right) \mathrm{d} x\,.
\end{equation*}
\end{lemma}

\begin{proof}
First, the definition \eqref{y21} of $\gamma_0$ and assumption~\eqref{gamma2} guarantee that there exist  $b>0$ and $s_b>s_0$ depending only on $\gamma$ such that, for $s\geq s_b$,
\begin{equation*}
\frac{1}{\gamma_0(s)} = \frac{\gamma^*}{\min\{1,\gamma^*\}\gamma(s)}\leq b s^k\,,
\end{equation*}
while the monotonicity of $\gamma_0$ ensures that
\begin{equation*}
\frac{1}{\gamma_0(s)}\leq \frac{1}{\gamma_0(s_b)} \le \frac{s^k}{B^kS_*^k \gamma_0(s_b)}\,, \qquad s\in [BS_*,s_b]\,.
\end{equation*}
Therefore, 
\begin{equation}\label{cond_gamma}
\frac{1}{\gamma_0(s)}\leq C s^{k}\,, \qquad s \in [B S_*,\infty)\,.
\end{equation}
Now, multiplying the inequality \eqref{e1a} by $\tilde{S}^{p-1}$ for some $p>1+k$ and integrating over $\Omega$, we obtain
\begin{equation*}
\frac{1}{p}\frac{d}{dt}\|\tilde{S}\|_p^p + B \int_\Omega (u+n)\gamma_0(\tilde{S})\tilde{S}^{p-1}\ \mathrm{d} x \leq \frac{\beta B}{A}\int_\Omega \tilde{S}^{p-1}\Gamma(\tilde{S})\ \mathrm{d} x + C_{\ref{cst11}} \|\tilde{S}\|_{p-1}^{p-1}\,.
\end{equation*}	
Thanks to \eqref{cond_gamma} and the fact $\tilde{S}=B S\geq B S_*$, see \eqref{lowbsw}, we derive  that
\begin{equation}
B \int_\Omega (u+n)\gamma_0(\tilde{S})\tilde{S}^{p-1}\ \mathrm{d} x \ge B C\int_\Omega \tilde{S}^{p-k-1}(u+n)\ \mathrm{d} x. \label{ile}
\end{equation}
Next, recalling that $\beta \tilde{S}-\Delta \tilde{S}=B(u+n)$ by the definition \eqref{defS} of $S$, we observe that
\begin{align*}
B \int_\Omega \tilde{S}^{p-k-1}(u+n)\ \mathrm{d} x & = \int_\Omega \tilde{S}^{p-k-1}(\beta \tilde{S}-\Delta \tilde{S})\ \mathrm{d} x \\
& = \beta \|\tilde{S}\|_{p-k}^{p-k} + (p-k-1) \int_\Omega \tilde{S}^{p-k-2} |\nabla \tilde{S}|^2\ \mathrm{d} x \\
& =  \beta \|\tilde{S}\|_{p-k}^{p-k} + \frac{4(p-k-1)}{(p-k)^2} \|\nabla \tilde{S}^{\frac{p-k}{2}}\|_2^2.
\end{align*}
Collecting the above estimates leads to the stated differential inequality and completes the proof.
\end{proof}

Next, we state the following lemma which can be proved  in the same way as  \cite[Lemma~4.4]{JiLau2021a}.

\begin{lemma}\label{lmrhs} \refstepcounter{NumConst}\label{cst14}
For any 
\begin{equation*}
p> \max\left\{\frac{2(N-1)}{N-2},l+\frac{N(k-l)}{2}\right\}
\end{equation*} 
and $\varepsilon>0$, there holds
\begin{equation}\label{rhs}
	\int_\Omega \left(\tilde{S}^{p-1}\Gamma(\tilde{S})+ \tilde{S}^{p-1}\right)\mathrm{d} x\leq \varepsilon\|\tilde{S}^{\frac{p-k}{2}}\|_{H^1}^2 + C_{\ref{cst14}}(\varepsilon,p)\,, \qquad t\ge 0\,.
\end{equation}
\end{lemma}

Then, in the same manner as in \cite[Lemma~4.5]{JiLau2021a}, we can derive from Lemma~\ref{lmp} and Lemma~\ref{lmrhs} the following uniform-in-time estimates for $\tilde{S}$ in $L^p(\Omega)$ for any $p\geq1$.

\begin{lemma}\label{lmvp} \refstepcounter{NumConst}\label{cst15}
	For any $p \geq 1$, there is $C_{\ref{cst15}}(p)>0$ such that
	\begin{equation*}
		\|\tilde{S}(t)\|_p\leq C_{\ref{cst15}}(p)\,, \qquad t\geq 0\,.
	\end{equation*}
\end{lemma}

With the above preparation, we are ready to use a Moser's iteration technique along the lines of \cite{Alik1979} to establish the boundedness of $\tilde{S}$ in  $L^\infty(\Omega)$ as stated in Proposition~\ref{propunib}, see also \cite{JiLau2021a}.

\begin{proof}[Proof of Proposition~\ref{propunib}]
We recall that \eqref{gvup}, \eqref{var0B}, and the elliptic comparison principle imply that
\begin{equation*}
\partial_t \tilde{S} + B\gamma_0(\tilde{S})(u+n) \le \beta B \gamma^* \mA^{-1}[u] + \beta B \mA^{-1}[n]\;\;\text{ in }\;\; (0,\infty)\times \Omega\,.
\end{equation*}
Since $\mA^{-1}[u]$ and $\mA^{-1}[n]$ are both non-negative and $S=\mA^{-1}[u+n]$ by \eqref{defS}, we may bound from above the right-hand side of the above inequality by $\beta B \max\{1,\gamma^*\} \mA^{-1}[u+n] = \beta \max\{1,\gamma^*\} \tilde{S}$, thus obtaining \refstepcounter{NumConst}\label{cst16}
\begin{equation*}
	\partial_t \tilde{S} + B\gamma_0(\tilde{S})(u+n) \le C_{\ref{cst16}} \tilde{S}\;\;\text{ in }\;\; (0,\infty)\times \Omega\,,
\end{equation*}
with $C_{\ref{cst16}} =\beta \max\{\gamma^*,1\}$. Then, for any $p \geq  2+k$, 
\begin{equation*}
\frac{d}{dt} \|\tilde{S}\|_p^p + p \int_\Omega (u+n) \gamma_0(\tilde{S}) \tilde{S}^{p-1} \mathrm{d} x \le C_{\ref{cst16}} p\|\tilde{S}\|_p^p.
\end{equation*}
We also infer from \eqref{defS} and \eqref{ile} that 
\begin{align*}
B \int_\Omega (u+n) \gamma_0(\tilde{S}) \tilde{S}^{p-1} \mathrm{d} x & \ge B C \int_\Omega (u+n) \tilde{S}^{p-k-1} \mathrm{d} x = C \int_\Omega (\beta \tilde{S}-\Delta \tilde{S}) \tilde{S}^{p-k-1} \mathrm{d} x \\
& = C \|\tilde{S}\|_{p-k}^{p-k} + 4C \frac{p-k-1}{(p-k)^2} \|\nabla \tilde{S}^{\frac{p-k}{2}} \|_2^2.
\end{align*}
Observing that
\begin{equation*}
\frac{p(p-k-1)}{(p-k)^2}\ge \frac{1}{2}
\end{equation*}
for $p\ge 2+k$ and combining the above inequalities give \refstepcounter{NumConst}\label{cst17}
\begin{equation}
\frac{d}{dt} \|\tilde{S}\|_p^p + C_{\ref{cst17}} \| \tilde{S}^{\frac{p-k}{2}} \|_{H^1}^2 \le C_{\ref{cst16}} p  \|\tilde{S}\|_p^p. \label{ami1}
\end{equation}

The remaining of the proof is then similar to the corresponding argument in the proof of \cite[Proposition~4.1]{JiLau2021a}. We nevertheless sketch it below for the sake of completeness. By H\"older's, Sobolev's, and Young's inequalities,
\begin{align*}
2 C_{\ref{cst16}} p\|\tilde{S}\|_p^p  & \le 2 C_{\ref{cst16}} p \|\tilde{S}^{\frac{p-k}{2}}\|_{\frac{2N}{N-2}} \|\tilde{S}\|_{\frac{N(p+k)}{N+2}}^{\frac{p+k}{2}} \le C p \|\tilde{S}^{\frac{p-k}{2}}\|_{H^1} \|\tilde{S}\|_{\frac{N(p+k)}{N+2}}^{\frac{p+k}{2}} \\
& \le C_{\ref{cst17}} \|\tilde{S}^{\frac{p-k}{2}}\|_{H^1}^2 + C p^2 \|\tilde{S}\|_{\frac{N(p+k)}{N+2}}^{p+k}. 
\end{align*}
Combining \eqref{ami1} and the above estimate, we find 
\begin{equation}
\frac{d}{dt} \|\tilde{S}\|_p^p + C_{\ref{cst16}}p \|\tilde{S}\|_p^p \le 2C_{\ref{cst16}} p \|\tilde{S}\|_p^p - C_{\ref{cst17}} \| \tilde{S}^{\frac{p-k}{2}} \|_{H^1}^2 \le C p^2 \|\tilde{S}\|_{\frac{N(p+k)}{N+2}}^{p+k}. \label{ami2}
\end{equation}

Let us now define two sequences $(p_j)_{j\ge 0}$ and $(X_j)_{j\ge 0}$ by 
\begin{equation}
\begin{split}
p_{j+1} & = \frac{N+2}{N} p_j - k, \qquad j\ge 0, \qquad p_0 = \frac{kN}{2} + \frac{2(N-1)}{N-2}, \\ 
X_j & = \sup_{t\ge 0} \|\tilde{S}(t)\|_{p_j}^{p_j}, \qquad j\ge 0.
\end{split} \label{ami3}
\end{equation}
Note that $X_j<\infty$ for all $j\ge 0$ due to Lemma~\ref{lmvp} since the choice of $p_0$ guarantees that
\begin{equation}
p_{j+1} > p_j > p_0 > \frac{kN}{2}\,, \quad j\ge 0\,,  \;\;\text{ and }\;\; \lim_{j\to \infty} p_j = \infty. \label{ami4}
\end{equation} 
For $j\ge 0$, we take $p=p_{j+1}$ in \eqref{ami2} and use \eqref{ami3} to obtain
\begin{equation*}
\frac{d}{dt} \|\tilde{S}\|_{p_{j+1}}^{p_{j+1}} + C_{\ref{cst16}} p_{j+1} \|\tilde{S}\|_{p_{j+1}}^{p_{j+1}} \le C p_{j+1}^2\|\tilde{S}\|_{p_j}^{\frac{(N+2) p_j}{N}} \le C p_{j+1}^2 X_j^{\frac{N+2}{N}}.
\end{equation*}
After an integration with respect to time, we find
\begin{align*}
\|\tilde{S}(t)\|_{p_{j+1}}^{p_{j+1}} & \le \|\tilde{S}^{in}\|_{p_{j+1}}^{p_{j+1}}  e^{-C_{\ref{cst16}} p_{j+1} t} + C p_{j+1} X_j^{\frac{N+2}{N}} \left( 1 - e^{-C_{\ref{cst16}} p_{j+1} t} \right) \\
& \le \max\left\{ |\Omega| \|\tilde{S}^{in}\|_\infty^{p_{j+1}} , C p_{j+1} X_j^{\frac{N+2}{N}} \right\}
\end{align*}
for all  $t\ge 0$, where $\tilde{S}^{in} = B S^{in}$. Hence,
\begin{equation*}
X_{j+1} \le C p_{j+1} \max\left\{ \|\tilde{S}^{in}\|_\infty^{p_{j+1}} , X_j^{\frac{N+2}{N}} \right\}, \qquad j\ge 0\,.
\end{equation*}
Owing to \eqref{ami3} and recalling that $X_{0}$ is finite by Lemma~\ref{lmvp},
we are in a position to apply Lemma~\ref{lmiter} with $\theta=\frac{N+2}{N}$, $b=1$, and $c=-k$ to conclude that the sequence $\left( X_j^{1/p_j} \right)_{j\ge 0}$ is bounded. Since
\begin{equation*}
	\|v(t)\|_\infty \le \|\tilde{S}(t)\|_\infty = \lim_{j\to\infty} \|\tilde{S}(t)\|_{p_j} \le \sup_{j\ge 0}\left\{ X_j^{1/j}\right\}\,, \qquad t\ge 0\,, 
\end{equation*}
by \eqref{y20}, we have shown the uniform-in-time boundedness of $v$.
\end{proof}

\subsection{Improved regularity}\label{sec.ir}

With the time-independent upper bound on $v$ derived in Proposition~\ref{propunib}, we may argue in the same manner as in Section~\ref{gecs} to derive the uniform-in-time boundedness of $(u,v,n)$. More precisely, we set $J=[0,\infty)$ and  regard $w$ as a solution to the following initial boundary value problem
\begin{subequations}\label{cpvu}
	\begin{align}
	\partial_t w + u \gamma(v) & = \beta\varphi+\psi\,, \qquad (t,x)\in (0,\infty)\times\Omega\,, \label{cpvu1} \\
	- \Delta w + \beta w & = u \,, \qquad (t,x)\in (0,\infty)\times\Omega\,, \label{cpvu1.5} \\
	\nabla w \cdot \nu & = 0\,, \qquad (t,x)\in (0,\infty)\times\partial\Omega\,, \label{cpvu2} \\
	w(0)  & = w^{in}\,, \qquad x\in\Omega\,, \label{cpvu3}
	\end{align}
\end{subequations}
the source term $\varphi$ and $\psi$ being still defined by \eqref{defvarphi} and \eqref{defpsi}, respectively.

In view of  \eqref{e00} and Proposition~\ref{propunib},  we have
\begin{equation}
0 < v_* \le v(t,x) \le v^*\;\;\text{ in }\;\; J\times \bar{\Omega}\,, \label{bvu}
\end{equation}
while \eqref{lowbsw}, \eqref{bvu}, and Propositions~\ref{Lemrev} and~\ref{Lemrevbis}  imply that
\begin{equation}\label{bwu}
	w_*\leq w(t,x)\leq w^*= \frac{v^*}{A}\;\;\text{ in }\;\; J\times \bar{\Omega}\,.
\end{equation}
Also, by \eqref{gvup}, \eqref{fup}, \eqref{bwu}, and the elliptic comparison principle,
\begin{equation}
0\leq \beta\varphi+\psi= \mA^{-1}[\beta u\gamma(v)+uf(n)]\leq \big( \beta \gamma^* + f^* \big) w \leq \big( \beta \gamma^* + f^* \big) w^* \;\;\text{ in }\;\; J\times\bar{\Omega}\,.\label{phibu}
\end{equation}
 Thanks to these properties, we may then argue as in the proof of Proposition~\ref{prop.impreg2} to derive the H\"older continuity of $v$.

\begin{proposition}\label{prop.impreg2u}
There is $\alpha\in (0,1)$ such that $v\in BUC^{\alpha}(J,C^\alpha(\bar{\Omega}))$.
\end{proposition}

 As in Section~\ref{sec.ir2}, we now exploit the H\"older regularity on $v$ provided by Proposition~\ref{prop.impreg2u} and, as in the proofs of Lemma~\ref{lem.impreg3}, we proceed along the lines of \cite[Section~6]{Aman1989} to establish the boundedness of the trajectory $\{w(t)\ :\ t\ge 0\}$ in $W^{2\theta,p}(\Omega)$ for all $p\in (N,\infty)$ and $\theta\in((N+p)/2p,1)$. We here take advantage of the validity of the needed bounds with constants which do not depend on time to derive estimates which also do not depend on time.

\begin{lemma} \label{lemma.impreg4u} \refstepcounter{NumConst}\label{cst18}
	For any  $p\in (N,\infty)$ and $\theta\in((N+p)/2p,1)$, there is $C_{\ref{cst18}}(p)>0$ such that
	\begin{equation*}
		\|w(t)\|_{W^{2\theta,p}} \le C_{\ref{cst18}}(p)\,, \qquad t\ge 0\,. 
	\end{equation*}
\end{lemma}

\begin{proof}  Recall that $J=[0,\infty)$. With the same notations as in Lemma~\ref{lem.impreg3}, we now have a unique parabolic fundamental solution $\tilde{U}$ associated to $\{ \mathsf{A}(t)\ :\ t\in J \}$ and there exist time-independent positive constants $M>0$, $M_\theta>0$, and $\omega>0$ such that 
	\begin{equation}
	\|\tilde{U}(t,s)\|_{\mathcal{L}(W^{2,p}(\Omega))} + 	\|\tilde{U}(t,s)\|_{\mathcal{L}(L^p(\Omega))} + (t-s) 	\|\tilde{U}(t,s)\|_{\mathcal{L}(L^p(\Omega),W^{2,p}(\Omega))} \le M e^{\omega(t-s)} \label{irb2u}
	\end{equation}
and
	\begin{equation}
	\|\tilde{U}(t,s)\|_{\mathcal{L}(W_{\mathcal{B}}^{2\theta,p}(\Omega))} + (t-s)^\theta \|\tilde{U}(t,s)\|_{\mathcal{L}(L^p(\Omega),W_{\mathcal{B}}^{2\theta,p}(\Omega))} \le M_\theta e^{\omega(t-s)} \label{irb3u}
	\end{equation}
	for $0\le s<t$. 
	
We then pick $\mu>\omega$ and deduce from \eqref{cpvu} that $w$ solves 
\begin{equation}
	\begin{split}
	\partial_t w + \left( \mu + \mathsf{A}(\cdot) \right)w  & = F+\mu w\,, \qquad t>0\,, \\
	w(0) & = w^{in} \,, 
	\end{split}\label{irb4u}
\end{equation} 
with $F=\beta \varphi+\psi - \beta w \gamma(v)$.
\refstepcounter{NumConst}\label{cst19} By \eqref{phibu} and the boundedness $v$ and $w$, there is $C_{\ref{cst19}}>0$ such that
\begin{equation}
	\|F(t)+\mu w(t)\|_\infty\leq C_{\ref{cst19}}\,, \qquad t\in J\,. \label{irb1000}
\end{equation}
 Owing the continuity of $u$, $v$, and $w$ provided by Theorem~\ref{TH1}, 
	\begin{equation*}
	F+\mu w \in C(J \times\bar{\Omega})\,, 
	\end{equation*}
	so that $w$ is the unique classical solution to \eqref{irb4u} on $J$. Using again \cite[Remarks~II.2.1.2~(a)]{Aman1995}, we conclude that $w$ is given by the representation formula 
	\begin{equation}
	w(t) = e^{-\mu t} \tilde{U}(t,0) w^{in} + \int_0^t e^{-\mu(t-s)} \tilde{U}(t,s) (F+\mu w)(s)\ \mathrm{d}s\,, \qquad t\ge 0\,. \label{irb9u}
	\end{equation}
	We then infer from \eqref{irb3u}, \eqref{irb1000}, and \eqref{irb9u} that, for $t\ge 0$,
	\begin{align}
	\|w(t)\|_{W^{2\theta,p}} & \le M_\theta e^{(\omega-\mu)t} \|w^{in}\|_{W^{2\theta,p}} +M_\theta \int_0^t (t-s)^{-\theta} e^{(\omega-\mu)(t-s)} \|(F+\mu w)(s)\|_p\ \mathrm{d}s  \nonumber \\
	& \le C(p,\theta) +  C_{\ref{cst19}} M_\theta |\Omega|^{1/p} \int_0^t (t-s)^{-\theta} e^{(\omega-\mu)(t-s)} \mathrm{d}s \nonumber \\
	&\leq C(p,\theta)\,,  \label{irb10u}
	\end{align}
since
	\begin{equation*}
	\mathcal{I}_\theta \triangleq \int_0^\infty s^{-\theta} e^{(\omega-\mu)s}\ \mathrm{d}s < \infty\,.
	\end{equation*}
\end{proof}

Last, we shall derive uniform-in-time $L^q$-estimates for $u$ and $q\in (1,\infty)$.

\begin{proposition}\label{prop.uLq} \refstepcounter{NumConst}\label{cst20} 
For any $q\in(1,\infty)$, there is $C_{\ref{cst20}}(q)>0$ such that 
\begin{equation*}
	\sup\limits_{t\geq0} \|u(t)\|_{q}\leq C_{\ref{cst20}}(q)\,.
\end{equation*}	
\end{proposition}

\begin{proof} 
We fix $t_*>0$ and first argue as in the beginning of the proof of Proposition~\ref{propLpu} to show that 
\begin{equation}
	\|\nabla v(t)\|_\infty \le C\,, \qquad t\in [t_*,\infty)\,, \label{irb14u}
\end{equation}
taking advantage here that the bounds we use do not depend on time and that the infinite integral with respect to time converges as it features the negative exponential $t\mapsto e^{-\beta t /2}$.

It remains to establish the uniform-in-time $L^q$-boundedness of $u$ for any $q>1$. Recalling the boundedness \eqref{fup} of $f(n)$ and the uniform-in-time boundedness of $\frac{|\gamma'(v)|^2}{\gamma(v)}$ on $[0,\infty)\times \bar{\Omega}$ which is a consequence of \eqref{g1} and \eqref{bvu}, we can argue in the same manner as in the proof of Proposition~\ref{propLpu} to obtain that, for $t\geq t_*$,\refstepcounter{NumConst}\label{cst21} 
\begin{align}\label{uniuest0}
	\frac{d}{dt}\|u\|_q^q+\frac{q(q-1)}{2}\int_\Omega \gamma(v)u^{q-2}|\nabla u|^2\rd x+ \|u\|_q^q \leq& \left( C(q) \sup\limits_{[t_*,\infty)}\|\nabla v\|_{\infty}+f^*+1 \right) \| u\|_q^q\nonumber\\
	\leq & C_{\ref{cst21}}(q) \| u\|_q^q\,.
\end{align}
Since $\gamma(v)\geq \gamma (v^*)$ due to \eqref{bvu} and the monotonicity \eqref{g1} of $\gamma$, we infer from Sobolev's embedding that \refstepcounter{NumConst}\label{cst22} 
\begin{align}
		\frac{q(q-1)}{2}\int_\Omega \gamma(v)u^{q-2}|\nabla u|^2\,\rd x + \|u\|_q^q & \geq\frac{2(q-1)\gamma(v^*)}{q} \left\|\nabla u^{q/2}\right\|_2^2 + \left\| u^{q/2} \right\|_2^2 \nonumber\\
		& \geq \min\left\{ \frac{2(q-1)\gamma(v^*)}{q} , 1 \right\} \left\| u^{q/2} \right\|_{H^1}^2 \non \\
		& \geq C_{\ref{cst22}}(q) \|u\|_{\frac{qN}{N-2}}^q\,. \label{y30}
\end{align}
It next follows from \eqref{uL1}, H\"older's and Young's inequalities that, for $\varepsilon>0$, 
\begin{equation*}
	\|u\|_q^q  \le \|u\|_{\frac{qN}{N-2}}^{\frac{Nq(q-1)}{N(q-1)+2}} \|u\|_1^{\frac{2q}{N(q-1)+2}} \le \|u\|_{\frac{qN}{N-2}}^{\frac{Nq(q-1)}{N(q-1)+2}} \|u^{in}+n^{in}\|_1^{\frac{2q}{N(q-1)+2}} \le \varepsilon \|u\|_{\frac{qN}{N-2}}^{q} + C(q,\varepsilon)\,.
\end{equation*}
Consequently, 
\begin{equation*}
	\|u\|_{\frac{qN}{N-2}}^{q} \ge \frac{\|u\|_q^q}{\varepsilon} - C(q,\varepsilon)\,,
\end{equation*}
which gives, together with \eqref{y30}, \refstepcounter{NumConst}\label{cst23}
\begin{equation} 
	\frac{q(q-1)}{2}\int_\Omega \gamma(v)u^{q-2}|\nabla u|^2\,\rd x + \|u\|_q^q \geq \frac{C_{\ref{cst22}}(q)}{\varepsilon} \|u\|_q^q - C_{\ref{cst23}}(q,\varepsilon)\,. \label{y31}
\end{equation}
Combining \eqref{uniuest0} and \eqref{y31} with $\varepsilon=C_{\ref{cst22}}(q)/2C_{\ref{cst21}}(q)$, we end up with
\begin{equation}\label{uestuuub}
	\frac{d}{dt}\| u\|_q^q + C_{\ref{cst21}}(q) \| u\|_q^q \leq C(q)\,,\qquad t\geq t_*\,,
\end{equation}
from which the uniform-in-time bounded of the $L^q$-norm of $u$ follows.
\end{proof}

 Theorem~\ref{TH2} is now a straightforward consequence of Proposition~\ref{prop.uLq} and a bootstrap argument.

\begin{proof}[Proof of Theorem~\ref{TH2}] 
With the aid of Proposition~\ref{prop.uLq}, we may further use a standard bootstrap argument (cf. \cite[Lemma~4.3]{AhnYoon2019}) to prove that
\begin{equation*}
\sup\limits_{t\ge 0} \|u(t)\|_{\infty}\leq C\,,
\end{equation*}
 and thus complete the proof of \eqref{ubd}. We then use the smoothing and decaying properties of the semigroup associated with the operator $\mathcal{A}$ to deduce from \eqref{cp2}, \eqref{cpn}, and \eqref{ubd} that both $v$ and $n$ belong to $L^\infty((t_0,\infty);C^1(\bar{\Omega}))$, while this property for $u$ follows from \eqref{cp1} and \cite[Theorem~5.2]{Aman1989}.
\end{proof}

\section{Stabilization toward the spatially homogeneous solution}

In this section, we study the stabilization of bounded classical solutions to \eqref{cp} toward the spatial homogeneous solution $(m,m,0)$ under the assumptions of Theorem~\ref{TH3}. More precisely, let $(u^{in},v^{in},n^{in})$ be initial conditions satisfying \eqref{ini} and assume that the corresponding non-negative classical solution $(u,v,n)$ to \eqref{cp} provided by Theorem~\ref{local} is global ($T_{\mathrm{max}}=\infty$) and bounded. Recalling \eqref{e00}, there are positive constants $u^*\ge 1$, $v_*>0$, and $v^*\ge 1$ such that 
\begin{equation}
	0 \le u(t,x) \le u^*\,, \qquad 0< v_* \le v(t,x) \le v^*\,, \qquad (t,x)\in [0,\infty)\times \Omega\,. \label{L03}
\end{equation} 
We next set 
\begin{equation*} 
	\gamma_1(s) \triangleq s \gamma(s) \;\;\text{ and }\;\; \Gamma_1(s) \triangleq \int_1^{s}\gamma_1(\eta)\rd \eta = \int_1^{s} \eta \gamma(\eta)\rd \eta\,, \qquad s>0\,.
\end{equation*}
Due to assumptions~\eqref{g1} and~\eqref{g_1'}, we have  
\begin{equation}
	\gamma' \le 0 \le \gamma_1'\,. \label{L00}
\end{equation}
We also denote the mean value of $z\in L^1(\Omega)$ by $\langle z\rangle$; that is,
	\begin{equation*}
		\langle z \rangle \triangleq \frac{1}{|\Omega|} \int_\Omega z(x)\ \mathrm{d}x\,.
	\end{equation*}
Introducing $a_* \triangleq \min\{1,v_*,\langle u^{in} \rangle\}>0$ and $a^* \triangleq \max\{u^*,v^*\}$, we infer from \eqref{uL1} and \eqref{L03} that
\begin{equation}
	\big(\langle u(t)\rangle , v(t) \big) \in [a_*,a^*]^2\,, \qquad t\in [0,\infty)\,. \label{y101}
\end{equation}

For simplicity, we only consider the case $\tau=\beta=1$ hereafter, the computations performed below being the same in the general case, while the derived identities feature additional factors depending on $\tau$ and $\beta$. Then $\kappa$ and $(\kappa_i)_{i\ge 1}$ denote positive constants depending only on $\Omega$, $\gamma$, $f$, the initial data, and the parameters $(u^*,v_*,v^*)$ introduced in \eqref{L03}.

\subsection{A Lyapunov functional}

 Building upon a construction designed in \cite{DLTW2021} for the system~\eqref{prlmod0}, we introduce in this section a Lyapunov functional, which plays a key role in the study of the long time behavior of classical solutions to \eqref{cp}. 

For $t\geq 0$, we define the unique solution $U(t)$ to the elliptic equation
\begin{equation}
	-\Delta U(t) = u(t) - \langle u(t) \rangle \;\;\text{ in }\;\; \Omega\,, \quad \nabla U(t)\cdot\nu = 0 \;\;\text{ on }\;\; \partial\Omega\,, \quad \langle U(t) \rangle = 0\,. \label{L01}
\end{equation}
Then, by \eqref{cp1} and \eqref{L01},
\begin{align}
 \frac{1}{2} \frac{d}{dt} \|\nabla U\|_2^2 & = - \int_\Omega U \partial_t \Delta U\ \mathrm{d}x = \int_\Omega U \partial_t u\ \mathrm{d}x \nonumber  \\
 & = \int_\Omega U \Delta (u\gamma(v))\ \mathrm{d}x + \int_\Omega U u f(n)\ \mathrm{d}x \nonumber \\
& = \int_\Omega u\gamma(v) \Delta U\ \mathrm{d}x + \int_\Omega U u f(n)\ \mathrm{d}x \nonumber\\
& = \int_\Omega (\langle u \rangle - u) u \gamma(v)\ \mathrm{d}x + \int_\Omega U u f(n)\ \mathrm{d}x\,. \label{y100}
\end{align}
To handle the first term on the right-hand side of \eqref{y100}, we proceed as in \cite{DLTW2021} to obtain
\begin{align*}
	J_0 & \triangleq \int_\Omega (\langle u \rangle - u) u \gamma(v)\ \mathrm{d}x \\
	& = \int_\Omega  (\langle u \rangle - 2v) u \gamma(v)\ \mathrm{d}x + \int_\Omega v \gamma_1(v)\ \mathrm{d}x - \int_\Omega (v-u)^2 \gamma(v)\ \mathrm{d}x\,.
\end{align*}
Since
\begin{align*}
	& \int_\Omega  (\langle u \rangle - 2v) u \gamma(v)\ \mathrm{d}x  = \int_\Omega  (\langle u \rangle - 2v) \gamma(v) (\partial_t v - \Delta v +v)\ \mathrm{d}x \\
	& \hspace{1cm} = \langle u \rangle \frac{d}{dt} \int_\Omega \Gamma(v)\ \mathrm{d}x - 2 \frac{d}{dt} \int_\Omega \Gamma_1(v)\ \mathrm{d}x   - \int_\Omega \left[ 2 \gamma_1'(v) - \langle u \rangle \gamma'(v) \right] |\nabla v|^2\ \mathrm{d}x\\
	& \hspace{1cm}\qquad + \langle u \rangle \int_\Omega \gamma_1(v)\ \mathrm{d}x - 2 \int_\Omega v \gamma_1(v)\ \mathrm{d}x\,,
\end{align*}
we obtain
\begin{align*}
	J_0 & = \langle u \rangle \frac{d}{dt} \int_\Omega \Gamma(v)\ \mathrm{d}x - 2 \frac{d}{dt} \int_\Omega \Gamma_1(v)\ \mathrm{d}x  - \int_\Omega \left[ 2 \gamma_1'(v) - \langle u \rangle \gamma'(v) \right] |\nabla v|^2\ \mathrm{d}x\\
	& \qquad - \int_\Omega (v-\langle u\rangle) \gamma_1(v)\ \mathrm{d}x - \int_\Omega (v-u)^2 \gamma(v)\ \mathrm{d}x\,.
\end{align*}
Consequently, inserting the above formula for $J_0$ in \eqref{y100} and using \eqref{cp1}, we obtain
\begin{align*}
	& \frac{d}{dt} \left[ \frac{\|\nabla U\|_2^2}{2} + 2 \int_\Omega \Gamma_1(v)\ \mathrm{d}x - \langle u \rangle \int_\Omega \Gamma(v)\ \mathrm{d}x \right] \\
	& \qquad = - \left( \int_\Omega \Gamma(v)\ \mathrm{d}x \right)  \frac{d}{dt} \langle u \rangle - \int_\Omega \left[ 2 \gamma_1'(v) - \langle u \rangle \gamma'(v) \right] |\nabla v|^2\ \mathrm{d}x\\
	& \qquad\qquad - \int_\Omega (v-\langle u\rangle) \gamma_1(v)\ \mathrm{d}x - \int_\Omega (v-u)^2 \gamma(v)\ \mathrm{d}x + \int_\Omega U u f(n)\ \mathrm{d}x \\
	& \qquad = - \int_\Omega \left[ 2 \gamma_1'(v) - \langle u \rangle \gamma'(v) \right] |\nabla v|^2\ \mathrm{d}x - \int_\Omega (v-\langle u\rangle) (\gamma_1(v)-\gamma_1(\langle u\rangle))\ \mathrm{d}x \\
	& \qquad\qquad - \int_\Omega (v-u)^2 \gamma(v)\ \mathrm{d}x + \int_\Omega U u f(n)\ \mathrm{d}x - \gamma_1(\langle u\rangle) \int_\Omega (v-\langle u\rangle)\ \mathrm{d}x \\
	& \qquad\qquad - \langle \Gamma(v) \rangle \| u f(n)\|_1\,.
\end{align*}
Finally, in view of \eqref{cp1} and \eqref{cp2}, we observe that
\begin{equation*}
	- \frac{d}{dt} \left( \gamma_1(\langle u\rangle) \|v\|_1 \right) = \gamma_1(\langle u\rangle)  \int_\Omega (v - \langle u\rangle)\ \mathrm{d}x -\gamma_1'(\langle u\rangle) \langle v\rangle \| u f(n)\|_1\,,
\end{equation*}
\begin{equation*}
	|\Omega| \frac{d}{dt} \bigg( \gamma_1(\langle u\rangle) \langle u\rangle - \Gamma_1(\langle u\rangle) \bigg) = \gamma_1'(\langle u\rangle)\langle u\rangle \frac{d}{dt}\|u\|_1 =\gamma_1'(\langle u\rangle)\langle u\rangle\|uf(n)\|_1,
\end{equation*}
and
\begin{equation*}
	|\Omega| \frac{d}{dt} \bigg( \langle u\rangle \Gamma(\langle u\rangle) - \Gamma_1(\langle u\rangle)\bigg)  = \Gamma(\langle u\rangle) \frac{d}{dt} \|u\|_1 = \Gamma(\langle u\rangle)\|uf(n)\|_1.
\end{equation*}

Gathering the above identities gives
\begin{align*}
	& \frac{d}{dt} \bigg[ \frac{\|\nabla U\|_2^2}{2} + 2 \int_\Omega \Gamma_1(v)\ \mathrm{d}x - \langle u \rangle \int_\Omega \Gamma(v)\ \mathrm{d}x - \gamma_1(\langle u\rangle) \| v\|_1\\
	&\qquad \qquad + |\Omega| \Big(\gamma_1(\langle u\rangle)\langle u\rangle+\langle u\rangle\Gamma(\langle u\rangle) - 2\Gamma_1(\langle u\rangle) \Big) \bigg] \\
	& \qquad  = - \int_\Omega \left[ 2 \gamma_1'(v) - \langle u \rangle \gamma'(v) \right] |\nabla v|^2\ \mathrm{d}x - \int_\Omega (v-\langle u\rangle) (\gamma_1(v)-\gamma_1(\langle u\rangle))\ \mathrm{d}x \\
	& \qquad\qquad - \int_\Omega (v-u)^2 \gamma(v)\ \mathrm{d}x + \int_\Omega U u f(n)\ \mathrm{d}x - \left[ \langle \Gamma(v) \rangle + \gamma_1'(\langle u\rangle) \langle v\rangle \right] \|uf(n)\|_1\,\\
	&\qquad\qquad +(\gamma_1'(\langle u\rangle)\langle u\rangle+\Gamma(\langle u\rangle))\|uf(n)\|_1\,.
\end{align*}

Let $K_*>0$ to be specified later. As
\begin{equation*}
	K_* \frac{d}{dt} \|n\|_1 = - K_* \|uf(n)\|_1\,,
\end{equation*}
we end up with
\begin{equation}
	\begin{split}
			& \frac{d}{dt} \bigg[ \frac{\|\nabla U\|_2^2}{2} + 2 \int_\Omega \Gamma_1(v)\ \mathrm{d}x - \langle u \rangle \int_\Omega \Gamma(v)\ \mathrm{d}x - \gamma_1(\langle u\rangle) \| v\|_1 + K_* \| n\|_1\\
	&\qquad \qquad + |\Omega| \bigg(\gamma_1(\langle u\rangle)\langle u\rangle+\langle u\rangle\Gamma(\langle u\rangle)-2\Gamma_1(\langle u\rangle)\bigg) \bigg] \\
	& \qquad  = - \int_\Omega \left[ 2 \gamma_1'(v) - \langle u \rangle \gamma'(v) \right] |\nabla v|^2\ \mathrm{d}x - \int_\Omega (v-\langle u\rangle) (\gamma_1(v)-\gamma_1(\langle u\rangle))\ \mathrm{d}x \\
	& \qquad\qquad - \int_\Omega (v-u)^2 \gamma(v)\ \mathrm{d}x + \int_\Omega U u f(n)\ \mathrm{d}x - \left[K_*+ \langle \Gamma(v) \rangle + \gamma_1'(\langle u\rangle) \langle v\rangle \right] \|uf(n)\|_1\,\\
	&\qquad\qquad +(\gamma_1'(\langle u\rangle)\langle u\rangle+\Gamma(\langle u\rangle))\|uf(n)\|_1\,.
	\end{split}\label{L02}
\end{equation}

It now follows from \eqref{L03}, Poincar\'e-Wirtinger's inequality
\begin{equation}
	\kappa_0 \|z-\langle z\rangle\|_2^2 \le \|\nabla z\|_2^2\,, \qquad z\in H^1(\Omega)\,, \label{PWE}
\end{equation} 
and Sobolev's embedding that, when $N\ge 3$,
\begin{equation*}
	\|U(t)\|_{2N/(N-2)} \le \kappa \|U(t)\|_{H^1}\le \kappa \|u-\langle u\rangle\|_2 \le \kappa\,, \qquad t\ge 0\,. 
\end{equation*}
As $U$ also solves $-\Delta U + U = U + u - \langle u \rangle$ in $\Omega$ with $\nabla U\cdot \nu=0$ on $\partial\Omega$, elliptic regularity provides a bound on $U$ in $W^{2,2N/(N-2)}(\Omega)$ and hence in $L^{2N/(N-6)_+}(\Omega)$. Thus, after a finite number of iterations, we conclude that there is $U^*>0$ depending only on $\Omega$ and $u^*$ such that
\begin{equation}
	|U(t,x)| \le U^*\,,  \qquad (t,x)\in [0,\infty)\times \Omega\,. \label{L04}
\end{equation}
A simpler argument gives \eqref{L04} when $N\in\{1,2\}$. We  now choose 
\begin{equation*}
	K_* \triangleq 1 + U^* + u^* \sup_{[a_*,a^*]}\{\gamma_1'\} + 2\Gamma(a^*) \ge 1\,,
\end{equation*}
so that, by \eqref{L03}, \eqref{L00}, \eqref{y101}, \eqref{L04}, and the monotonicity of $\Gamma_1$,
\begin{align*}
	&\int_\Omega U u f(n)\ \mathrm{d}x - \bigg[ K_* + \langle \Gamma(v) \rangle + \gamma_1'(\langle u\rangle) \langle v\rangle -\gamma_1'(\langle u\rangle)\langle u\rangle-\Gamma(\langle u\rangle)\bigg] \|uf(n)\|_1\\
	 & \qquad\qquad \le \left[  U^* + u^* \sup_{[a_*,a^*]}\{\gamma_1'\} + 2\Gamma(a^*) - K_*  \right] \|uf(n)\|_1 \\
	 &\qquad\qquad \le - \|uf(n)\|_1\,.
\end{align*}
We then infer from \eqref{L00}, \eqref{L02}, and the above estimate that
\begin{equation}
\frac{d}{dt} L + D_1 + D_2 + D_3 + D_4\le 0 \,, \label{L05}
\end{equation}
where
\begin{align*}
	L & \triangleq \frac{\|\nabla U\|_2^2}{2} + 2 \int_\Omega \Gamma_1(v)\ \mathrm{d}x - \langle u \rangle \int_\Omega \Gamma(v)\ \mathrm{d}x - \gamma_1(\langle u\rangle) \| v\|_1 + K_* \| n\|_1\\
	&\qquad + |\Omega| \bigg(\gamma_1(\langle u\rangle)\langle u\rangle+\langle u\rangle\Gamma(\langle u\rangle)-2\Gamma_1(\langle u\rangle)\bigg)\,,  \\
	D_1 & \triangleq \int_\Omega \left[ 2 \gamma_1'(v) - \langle u \rangle \gamma'(v) \right] |\nabla v|^2\ \mathrm{d}x \ge 0\,, \\
	D_2 & \triangleq \int_\Omega (v-\langle u\rangle) (\gamma_1(v)-\gamma_1(\langle u\rangle))\ \mathrm{d}x \ge 0\,, \\
	D_3 & \triangleq  \int_\Omega (v-u)^2 \gamma(v)\ \mathrm{d}x \geq\gamma(v^*)  \int_\Omega (v-u)^2 \ \mathrm{d}x \ge 0\,, \\
	D_4 & \triangleq  \|uf(n)\|_1\ge 0\,.
\end{align*}
We also note that the convexity of $\Gamma_1$ and $-\Gamma$, see  \eqref{L00}, and the non-negativity of $\langle u\rangle$ guarantee that
\begin{align}
	L & = \frac{\|\nabla U\|_2^2}{2} + K_* \| n\|_1 + 2 \int_\Omega \big( \Gamma_1(v) - \Gamma_1(\langle u\rangle) \big)\ \mathrm{d}x  \nonumber\\ 
	& \qquad + \langle u\rangle \int_\Omega \big( - \Gamma(v) + \Gamma(\langle u\rangle) \big)\ \mathrm{d}x - \gamma_1(\langle u\rangle ) \int_\Omega \big( v - \langle u\rangle \big)\ \mathrm{d}x \nonumber \\
	& \ge \frac{\|\nabla U\|_2^2}{2} + K_* \| n\|_1 + 2 \int_\Omega \gamma_1(\langle u\rangle) \big( v-\langle u\rangle \big)\ \mathrm{d}x - \langle u\rangle \gamma(\langle u\rangle) \int_\Omega \big( v - \langle u\rangle \big)\ \mathrm{d}x \nonumber \\
	& \qquad - \gamma_1(\langle u\rangle ) \int_\Omega \big( v - \langle u\rangle \big)\ \mathrm{d}x \nonumber \\
	& \ge \frac{\|\nabla U\|_2^2}{2} + K_* \| n\|_1 \ge 0\,. \label{L06}
\end{align}
A first consequence of \eqref{L05} and \eqref{L06} is that
\begin{equation}
	\int_0^\infty \left( \sum_{i=1}^4 D_i(s) \right)\ \mathrm{d}s \le L(0)  < \infty\,. \label{L07}
\end{equation}
Next, since $\langle u(t)\rangle\geq \langle u^{in}\rangle>0$ by \eqref{uL1} and $-\gamma'\ge 0$ by \eqref{L00}, we notice that
\begin{equation*}
	D_1\geq \int_\Omega \left[ 2 \gamma_1'(v) - \langle u^{in} \rangle \gamma'(v) \right] |\nabla v|^2\ \mathrm{d}x.
\end{equation*} 
We furthermore observe that $2 \gamma_1'(s) - \langle u^{in} \rangle \gamma'(s) >0$ for all $s>0$. Otherwise, there is $s_0>0$ such that $2 \gamma_1'(s_0) - \langle u^{in} \rangle \gamma'(s_0) =0$ and it follows that $\gamma_1'(s_0)=\gamma'(s_0) = 0$ due to the non-negativity \eqref{L00} of $\gamma_1'$ and $-\gamma'$. Thus, $0=\gamma_1'(s_0)=s_0\gamma'(s_0)+\gamma(s_0)=\gamma(s_0)$, which contradicts the positivity of $\gamma$ in $(0,\infty)$. Consequently, there exists $\kappa_1>0$ such that 
\begin{equation}\label{V09}
	D_1\geq \int_\Omega \left[ 2 \gamma_1'(v) - \langle u^{in} \rangle \gamma'(v) \right] |\nabla v|^2\ \mathrm{d}x\geq \kappa_1\|\nabla v\|_2^2\,.
\end{equation} 
Hence, by \eqref{L07},
\begin{equation}\label{L19}
	\int_0^\infty \|\nabla v(t)\|^2_2\ \rd t<\infty\,.
\end{equation}

Next, by \eqref{cp2}, \eqref{L03}, \eqref{L00}, and Cauchy-Schwarz' inequality, 
\begin{align}
	\|\partial_t v\|_2^2 + \frac{1}{2} \frac{d}{dt} \|\nabla v\|_2^2 & = \int_\Omega (u-v) \partial_t v\ \mathrm{d}x \label{L200} \\
	& \le \frac{1}{2} \|\partial_t v\|_2^2 + \frac{1}{2} \|u-v\|_2^2  \le \frac{1}{2} \|\partial_t v\|_2^2 + \frac{D_3}{2\gamma(v^*)} \,. \nonumber
\end{align}
Hence,
\begin{equation}
	\|\partial_t v\|_2^2 + \frac{d}{dt} \|\nabla v\|_2^2 \le \frac{D_3}{\gamma(v^*)}\,. \label{L08}
\end{equation}
We then infer from \eqref{L07} and \eqref{L08} that
\begin{equation}
	\sup_{t\ge 0} \{\|\nabla v(t)\|_{H^1}\} + \int_0^\infty \|\partial_t v(s)\|_2^2\ \mathrm{d}s < \infty\,. \label{L09}
\end{equation}

\subsection{Stabilization toward the spatially homogeneous solution} 

We are now in a position to study the large time behaviour of $(u,v,n)$ and first exploit the previous analysis to prove the convergence of $(u(t),v(t),n(t))$ to the spatially homogeneous steady  state $(m,m,0)$ as $t\to\infty$.

\begin{lemma}\label{lmstability} 
Recalling that $m=\langle u^{in} + n^{in} \rangle$, there holds
	\begin{equation*}
	\lim\limits_{t\rightarrow\infty}\bigg(	\|u(t)-m\|_{\infty}+\|v(t)-m\|_{\infty}+\|n(t)\|_{\infty}\bigg)=0.
	\end{equation*}
\end{lemma}

\begin{proof}
By Theorem~\ref{local},
\begin{equation}
	m = \langle u(t)+n(t) \rangle\,, \qquad t\ge 0\,, \label{p4}
\end{equation}
and
\begin{equation}
	m_u \triangleq \lim_{t\to \infty} \langle u(t)\rangle \ge \langle u^{in}\rangle>0\,, \qquad m_n \triangleq \lim_{t\to \infty} \langle n(t)\rangle \in [0,\langle n^{in}\rangle]\,. \label{p5} 
\end{equation}
Thanks to \eqref{cp2} and \eqref{p5}, we also see that
\begin{equation}
	\lim_{t\to\infty} \langle v(t) \rangle = m_u\,. \label{L10}
\end{equation}

It next follows \eqref{L07}, \eqref{L09}, and Cauchy-Schwarz' inequality that
\begin{align*}
	\int_0^\infty \|(u-v)(t) \partial_t v(t)\|_1\ \mathrm{d}t & \le \frac{1}{2} \int_0^\infty \|(u-v)(t)\|_2^2\ \mathrm{d}t + \frac{1}{2}  \int_0^\infty \|\partial_t v(t)\|_2^2\ \mathrm{d}t  \\
	& \le \frac{1}{2\gamma(v^*)} \int_0^\infty D_3(t)\ \mathrm{d}t + \frac{1}{2}  \int_0^\infty \|\partial_t v(t)\|_2^2\ \mathrm{d}t < \infty\,.
\end{align*}
Consequently, $(u-v) \partial_t v$ belongs to $L^1((0,\infty)\times\Omega)$ and it follows from \eqref{L200} and \eqref{L09} that
\begin{equation*}
	V_\infty \triangleq \lim_{t\to\infty} \|\nabla v(t)\|_2^2 =  \|\nabla v^{in}\|_2^2 + 2 \int_0^\infty \int_\Omega (u-v)(t) \partial_t v(t)\ \mathrm{d}x\mathrm{d}t - 2 \int_0^\infty \|\partial_t v(t)\|_2^2\ \mathrm{d}t
\end{equation*}
is well-defined and finite. But \eqref{L19} implies that $V_\infty=0$ and we end up with
\begin{equation}
	\lim_{t\to\infty} \|\nabla v(t)\|_2 = 0\,. \label{L201}
\end{equation}
We now combine \eqref{L10}, \eqref{L201}, and Poincar\'e-Wirtinger's inequality \eqref{PWE} to obtain
\begin{align*}
	\lim_{t\to\infty} \|v(t)-m_u\|_2 & \le \lim_{t\to\infty} \left[ \|v(t)-\langle v(t)\rangle\|_2 + \sqrt{|\Omega|} |\langle v(t)\rangle - m_u| \right] \\
	& \le \frac{1}{\sqrt{\kappa_0}} \lim_{t\to\infty} \|\nabla v(t)\|_2 = 0\,.
\end{align*}
Together with the boundedness \eqref{L03} of $v$ in $L^\infty(\Omega)$, the above convergence implies that
\begin{equation}
	\lim_{t\to\infty} \|v(t)-m_u\|_p=0\,, \qquad p\in [1,\infty)\,. \label{p6}
\end{equation}

Next, according to \eqref{cpn}
\begin{align*}
	\frac{1}{2} \frac{d}{dt} \|n-\langle n\rangle\|_2^2 & = \int_\Omega (n-\langle n\rangle) \partial_t n\ \mathrm{d}x = - \|\nabla n\|_2^2 - \int_\Omega (n-\langle n\rangle) u f(n)\ \mathrm{d}x \\
	& \le - \|\nabla n\|_2^2 + \langle n\rangle \|uf(n)\|_1\,.
\end{align*}
 It then follows from \eqref{nup} (with $p=1$) and Poincar\'e-Wirtinger's inequality \eqref{PWE} that
\begin{equation}\label{N01}
	\frac{d}{dt} \|n-\langle n\rangle\|_2^2 + \kappa_0 \|n-\langle n\rangle\|_2^2 + \|\nabla n\|_2^2\le  2 \langle n\rangle \|uf(n)\|_1 \le 2 \langle n^{in} \rangle \|u f(n)\|_1\,.
\end{equation}
Integrating \eqref{N01} and using \eqref{L07}, we end up with
\begin{equation*}
	\lim_{t\to\infty} \|n(t) - \langle n(t)\rangle\|_2 = 0\,,
\end{equation*} 
which, together with \eqref{p5}, yields that
\begin{equation}
	\lim_{t\to\infty} \|n(t) - m_n\|_2  = 0\,. \label{p9}
\end{equation} 

Next, by \eqref{L03}, \eqref{L00}, and \eqref{L07},
\begin{equation*}
	\gamma(v^*) \int_0^\infty \|(v-u)(s)\|_2^2\ \mathrm{d}s  \le \int_0^\infty D_3(s) \ \mathrm{d}s < \infty\,,
\end{equation*}
so that
\begin{equation}
	\lim_{t\to\infty} \int_t^{t+1} \|(v-u)(s)\|_2^2\ \mathrm{d}s = 0\,. \label{p10}
\end{equation}
Consequently, by \eqref{nup},
\begin{align*}
	\int_t^{t+1} \|v(s)f(n(s))\|_1\ \mathrm{d}s & \le \int_t^{t+1} \|(v-u)(s)f(n(s))\|_1\ \mathrm{d}s + \int_t^{t+1} \|u(s)f(n(s))\|_1\ \mathrm{d}s \\
	& \le |\Omega|^{1/2}\sup_{[0,\|n^{in}\|_\infty]}\{f\} \left( \int_t^{t+1} \|(v-u)(s)\|_2^2\ \mathrm{d}s \right)^{1/2} \\
	& \qquad + \int_t^{t+1} \|u(s)f(n(s))\|_1\ \mathrm{d}s \,.
\end{align*}
Since the right-hand side of the above inequality converges to zero as $t\to\infty$ by \eqref{L07} and \eqref{p10}, we conclude that 
\begin{equation*}
	\lim_{t\to\infty} \int_t^{t+1} \|v(s)f(n(s))\|_1\ \mathrm{d}s = 0\,.
\end{equation*}
Since \eqref{p6} and \eqref{p9} entail also a.e. convergences, we deduce from Fatou's lemma that
\begin{equation*}
	m_u f(m_n) = 0\,.
\end{equation*}
Thus,  $m_n=0$ and $m_u=m$ in view of the positivity assumption $f(s)>f(0)=0$ for all $s>0$.

It remains to show the convergence of $u$ to $m$. Since 
\begin{equation*}
	\partial_t(u-\la u\ra)=\Delta(u\gamma(v))+uf(n)-\la uf(n)\ra \;\;\text{ in }\;\; (0,\infty)\times\Omega
\end{equation*}
by \eqref{cp1}, we deduce from \eqref{L03}, \eqref{L00}, and Young's inequality that
\begin{align*}
	\frac{d}{dt}\|u-\la u\ra\|_2^2&=-2\int_\Omega \gamma(v)|\nabla u|^2\ \rd x -2\int_\Omega u\gamma'(v)\nabla u\cdot \nabla v \ \rd x  \\
	& \qquad\qquad +2 \int_\Omega (u-\la u\ra)(uf(n)-\la uf(n)\ra)\ \rd x \\
	&\leq -\int_\Omega \gamma(v)|\nabla u|^2\ \rd x+\int_\Omega \frac{|\gamma'(v)|^2}{\gamma(v)}u^2|\nabla v|^2 \ \rd x+  2\int_\Omega (u-\la u\ra)uf(n)\ \rd x \\
	&\leq -\gamma(v^*)\|\nabla u\|_2^2 + \kappa \|\nabla v\|^2_2 + \kappa D_4\,.
\end{align*}
A further use of Poincar\'e-Wirtinger's inequality \eqref{PWE} gives
\begin{equation}
	\frac{d}{dt}\|u-\la u\ra\|_2^2 \leq -\frac{\gamma(v^*)}{2}\|\nabla u\|_2^2 - \frac
	{\kappa_0\gamma(v^*)}{2}\|u-\la u\ra\|_2^2 + \kappa \left( \|\nabla v\|^2_2 + D_4 \right)\,.\label{U02}
\end{equation}
Then, the above inequality, together with \eqref{L07}, \eqref{L19}, and \eqref{p5} yields that
\begin{equation}\label{U03}
	\lim\limits_{t\rightarrow\infty}\|u(t)-m\|_2=0.
\end{equation}

Finally, the convergence of $(u,v,n)$ toward $(m,m,0)$ in $L^p(\Omega;\mathbb{R}^3)$ for all $p\in[1,\infty)$ readily follows from \eqref{nup}, \eqref{L03}, \eqref{p6}, \eqref{p9}, and \eqref{U03}. Making use of smoothing and decaying properties of parabolic equations as outlined at the end of the proof of Theorem~\ref{TH2}, the trajectory $\{(u(t),v(t),n(t))\ :\ t\ge 1\}$ is actually bounded in $C^1(\bar{\Omega})$, from which the convergence of $(u,v,n)$ toward $(m,m,0)$ in $L^\infty(\Omega;\mathbb{R}^3)$ follows by an interpolation argument.
\end{proof}

\subsection{The exponential stabilization} 

In this section, we prove the exponential stabilization under the addition assumption 
\begin{equation}\label{asf2}
	f_0\triangleq\liminf\limits_{s\rightarrow0+}\frac{f(s)}{s}>0\,,
\end{equation}
bearing in mind that $f>0$ on $(0,\infty)$ and satisfies \eqref{asf}. Owing to Lemma~\ref{lmstability} and the above assumption, we may now assume there is $t_0>0$ such that  the trajectory  $\{(u(t), v(t), n(t))\ :\ t\ge t_0\}$ is  sufficiently close to the spatially homogeneous solution $(m,m,0)$; that is,
\begin{equation}\label{U06}
	\|u(t)-m\|_{\infty}\leq \frac{m}{2}\,, \qquad t\ge t_0\,,
\end{equation}
and
\begin{equation}\label{N002}
	f(n(t,x))\geq \frac{f_0}{2}n(t,x)\,, \qquad (t,x)\in [t_0,\infty)\times \bar{\Omega}\,.
\end{equation}

\begin{lemma}\label{lmexpstability}
There exist $\kappa_2>0$ and $\kappa_3>0$ such that, for all $t\geq t_0,$
	\begin{equation*}
		\|u(t)-m\|_{\infty} + \|v(t)-m\|_{\infty} + \|n(t)\|_{\infty} \leq \kappa_2 e^{-\kappa_3 t}\,.
	\end{equation*}
\end{lemma}

\begin{proof}
	In order to show the exponential stabilization, we need to establish some control on the Lyapunov functional by the dissipation terms. First, by \eqref{cp2}, \eqref{L01},  Poincar\'e-Wirtinger's inequality \eqref{PWE}, \eqref{V09}, and Young's inequality,
\begin{align}\label{U04}
	\|\nabla U\|_2^2\leq& \|u-\langle u\rangle\|_2^2\le \kappa \left( \|u-v\|_2^2 + \|v-\langle v\rangle\|_2^2 + \|\langle v\rangle - \langle u\rangle\|_2^2 \right) \non\\
	\leq &\kappa \left( \|u-v\|_2^2+\|\nabla v\|_2^2+\|\langle \partial_t v \rangle\|_2^2 \right)\non\\
	 	 \leq&  \kappa \left( D_1+D_3+\|\partial_t v\|_2^2 \right)\,.
\end{align}

Next, since \begin{equation*}
	\Gamma_1(v)=\Gamma_1(\langle u\rangle)+\gamma_1(\langle u\rangle)(v-\langle u\rangle)+\frac12\gamma_1'(\xi_1)(v-\langle u\rangle)^2
\end{equation*}and
\begin{equation*}
	\Gamma(v)=\Gamma(\langle u\rangle)+\gamma(\langle u\rangle)(v-\langle u\rangle)+\frac12\gamma'(\xi_2)(v-\langle u\rangle)^2
\end{equation*} 
with some $\xi_i$, $i=1,2$, depending on $(t,x)$ and lying between $v$ and $\la u\ra$, we notice that, using again \eqref{cp2}, as well as the bounds \eqref{y101} and the regularity \eqref{g1} of $\gamma$,
\begin{align*}
	\Gamma_1(v)-\Gamma_1(\langle u\rangle)+\gamma_1(\langle u\rangle)\langle u\rangle-\gamma_1(\langle u\rangle)v & =\frac{1}{2}\gamma_1'(\xi_1)(v-\langle u\rangle)^2\\
	&\leq \sup_{[a_*,a^*]}\{\gamma_1'\} \left[ (v-\langle v\rangle)^2+(\langle v\rangle-\langle u\rangle)^2 \right]\\
	&\leq \kappa \left[ (v-\langle v\rangle)^2 + |\langle \partial_t v\rangle|^2 \right] \,.
\end{align*}
Similarly, \begin{align*}
	\Gamma_1(v)-\Gamma_1(\langle u\rangle)-\langle u\rangle\Gamma(v)+\langle u\rangle\Gamma(\langle u\rangle) &=\frac{1}{2} \left( \gamma_1'(\xi_1) - \langle u\rangle \gamma'(\xi_2) \right) (v-\langle u\rangle)^2 \\
	&\leq \kappa \left[ (v-\langle v\rangle)^2 + |\langle \partial_t v\rangle|^2 \right] \,.
\end{align*}

Consequently, due to \eqref{V09} and Poincar\'e-Wirtinger's inequality \eqref{PWE},
\begin{align}\label{Gamma009}
	 2 \int_\Omega \Gamma_1(v)\ \mathrm{d}x & - \langle u \rangle \int_\Omega \Gamma(v)\ \mathrm{d}x - \gamma_1(\langle u\rangle) \| v\|_1 + |\Omega| \bigg(\gamma_1(\langle u\rangle)\langle u\rangle+\langle u\rangle\Gamma(\langle u\rangle)-2\Gamma_1(\langle u\rangle)\bigg)\nonumber\\
	 & = \int_\Omega \left[ \Gamma_1(v)-\Gamma_1(\langle u\rangle)+\gamma_1(\langle u\rangle)\langle u\rangle-\gamma_1(\langle u\rangle)v \right]\ \mathrm{d}x  \nonumber \\
	 & \qquad + \int_\Omega \left[ \Gamma_1(v)-\Gamma_1(\langle u\rangle)-\langle u\rangle\Gamma(v)+\langle u\rangle\Gamma(\langle u\rangle) \right]\ \mathrm{d}x  \nonumber \\
	 & \leq \kappa \int_\Omega \left[ (v-\langle v\rangle)^2 + |\langle \partial_t v\rangle|^2 \right]\ \mathrm{d}x \nonumber\\
	 &\leq \kappa \left( \|\nabla v\|_2^2+\|\partial_t v\|_2^2 \right) \leq \kappa \left( D_1 + \|\partial_t v\|_2^2 \right)\,.
\end{align}
	
Next, \eqref{U06} and \eqref{N002} indicate that, for all $t\geq t_0$,
\begin{align}\label{N03}
	D_4 = \|uf(n)\|_1\geq \frac{m}{2}\|f(n)\|_1\geq \frac{mf_0}{4}\|n\|_1\,.
\end{align}

Thus, it follows from \eqref{U04}, \eqref{Gamma009}, and \eqref{N03} that, for $t\geq t_0$,
\begin{equation*}
	L \leq \kappa_4 \left( D_1+D_3+D_4 + \|\partial_t v\|_2^2 \right)\,,
\end{equation*}	
with $\kappa_4\ge 1$. Combining the above estimate with \eqref{L05}, \eqref{V09}, and \eqref{L08}, we obtain, for $\delta_1\in (0,1)$, $\delta_2\in (0,1)$, and $t\ge t_0$,
\begin{align*}
	&\frac{d}{dt}(L+\delta_1\|\nabla v\|_2^2)+\delta_2 (L+\delta_1\|\nabla v\|_2^2)+\sum_{i=1}^4 D_i+\delta_1\|\partial_t v\|_2^2\\
	&\qquad\leq \delta_2 (L+\delta_1\|\nabla v\|_2^2) +\frac{\delta_1}{\gamma(v^*)}D_3\\
	&\qquad\leq \delta_2 \kappa_4 \left( D_1+D_3+D_4 + \|\partial_t v\|_2^2 \right) + \frac{\delta_1 \delta_2}{\kappa_1} D_1  + \frac{\delta_1}{\gamma(v^*)}D_3 \\
	& \qquad \leq \left( \delta_2 \kappa_4 + \frac{\delta_1}{\kappa_1} + \frac{\delta_1}{\gamma(v^*)} \right) \sum_{i=1}^4 D_i + \delta_2 \kappa_4 \|\partial_t v\|_2^2\,.
\end{align*}
Then, picking $\delta_1\in (0,1)$ and $\delta_2\in (0,1)$ such that
\begin{equation*}
	\delta_2 \triangleq \frac{\delta_1}{2 \kappa_4} \;\;\text{ and } \left( \frac{1}{2} + \frac{1}{\kappa_1} + \frac{1}{\gamma(v^*)} \right) \delta_1 \triangleq \frac{1}{2}\,,
\end{equation*}
we finally arrive at
\begin{equation}
	\frac{d}{dt}(L+\delta_1\|\nabla v\|_2^2)+\delta_2 (L+\delta_1\|\nabla v\|_2^2)\leq 0, \qquad t\ge t_0\,.
\end{equation}
Integrating the above differential inequality leads us to
\begin{equation*}
	L(t)+\delta_1\|\nabla v(t)\|_2^2\leq (L(t_0)+\delta_1\|\nabla v(t_0)\|_2^2) e^{-\delta_2(t-t_0)}\,\qquad\text{for all}\;t\geq t_0.
\end{equation*}
Recalling \eqref{L06}, we obtain 
\begin{equation}\label{expsta01}
	\|\nabla U(t)\|_2^2+\|\nabla v(t)\|_2^2+\|n(t)\|_1\leq \kappa e^{-\delta_2 t}\,, \qquad t\ge t_0\,.
\end{equation}

We next deduce from \eqref{N01}, \eqref{U02}, and the definition of $D_4$ that
\begin{align*}
	\frac{d}{dt} \bigg( \|u-\langle u\rangle\|_2^2 + \|n-\langle n\rangle\|_2^2 \bigg) & \le - \frac{\gamma(v^*)}{2} \|\nabla u\|_2^2 - \frac{\kappa_0 \gamma(v^*)}{2} \|u-\langle u\rangle\|_2^2 + \kappa \left( \|\nabla v\|_2^2 + \|u f(n)\|_1 \right) \\
	& \qquad - \|\nabla n\|_2^2 - \kappa_0 \|n-\langle n\rangle\|_2^2 + 2\langle n \rangle \|u f(n)\|_1\,.
\end{align*}
Since \begin{equation*}
	\kappa \|uf(n)\|_1+2\langle n\rangle \|u f(n)\|_1 \le (\kappa+2m) u^* \max_{[0,\|n^{in}\|_\infty]}|f'| \|n\|_1 \le \kappa \|n\|_1
\end{equation*}
by \eqref{asf}, \eqref{e0}, \eqref{nup} (with $p=\infty$), and \eqref{L03}, we infer from the above estimates  that
\begin{align*}
	\frac{d}{dt} \bigg( \|u-\langle u\rangle\|_2^2 + \|n-\langle n\rangle\|_2^2 \bigg) & \le  - \frac{\kappa_0 \gamma(v^*)}{2} \|u-\langle u\rangle\|_2^2  + \kappa \left(  \|\nabla v\|_2^2 + \|n\|_1 \right) - \kappa_0 \|n-\langle n\rangle\|_2^2 \\
	& \le - \delta_3 \bigg( \|u-\langle u\rangle\|_2^2 + \|n-\langle n\rangle\|_2^2 \bigg) + \kappa \left( \|\nabla v\|_2^2 + \|n\|_1 \right)\,,
\end{align*}
with $2\delta_3 \triangleq \kappa_0 \min\{\gamma(v^*), 2\}>0$. Integrating the above differential inequality and using \eqref{expsta01} and Poincar\'e-Wirtinger's inequality \eqref{PWE} give
\begin{equation}
\|u(t) - \langle u(t)\rangle\|_{2}^2 + \|v(t)-\langle v(t)\rangle\|_2^2 + \|n(t)-\langle n(t)\rangle\|_2^2 \leq Ce^{-\min\{\delta_2,\delta_3\} t}\qquad\text{for all}\;\; t\geq t_0\,. \label{E01}
\end{equation}

Finally, by \eqref{cp2}, \eqref{e0}, and \eqref{expsta01},
\begin{align*}
	& |\langle n(t)\rangle| = \frac{\|n(t)\|_1}{|\Omega|} \le \kappa e^{-\delta_2 t}\,, \qquad t\ge t_0\,, \\
	& |\langle u(t) \rangle - m| = |\langle n(t)\rangle| \le \kappa e^{-\delta_2 t}\,, \qquad t\ge t_0\,, 
\end{align*}
and 
\begin{align*}
	|\langle v(t) \rangle - m | & = \left|  \big( \langle v(t_0) \rangle  - m \big) e^{t_0-t} + \int_{t_0}^t \big( \langle u(s) \rangle - m \big) e^{s-t}\ \mathrm{d}s \right| \\
	& \le |\langle v(t_0) \rangle - m| e^{t_0-t} + \int_{t_0}^t |\langle u(s) \rangle - m| e^{s-t}\ \mathrm{d}s  \\
	& \le \big( \langle v(t_0) \rangle + m \big)  e^{t_0-t} + \kappa e^{-\delta_2 t}\,, \qquad t\ge t_0\,,
\end{align*}
recalling that $0<\delta_2<1$. Together with \eqref{E01}, the above estimates entail that
\begin{equation}
	\|u(t) - m\|_{2}^2 + \|v(t)-m\|_2^2 + \|n(t)\|_2^2 \leq Ce^{-\delta_4 t}\qquad\text{for all}\;\; t\geq t_0\,, \label{E02}
\end{equation}
with $\delta_4 \triangleq \min\{\delta_2,\delta_3\}\in (0,1)$.

With the exponential convergence \eqref{E02} at hand, we can further improve the exponential convergence to more regular spaces via standard bootstrap argument. A similar procedure can be found in \cite{Jiang2020}. Hence we omit the details here.
\end{proof}

\begin{proof}[Proof of Theorem~\ref{TH3}] 
Theorem~\ref{TH3} follows directly from Lemma~\ref{lmstability} and Lemma~\ref{lmexpstability}.
\end{proof}

\section*{Acknowledgments}
J.~Jiang  is supported by Hubei Provincial Natural Science Foundation under the grant No. 2020CFB602. Ph.~Lauren\c{c}ot thanks Ariane Trescases for helpful discussions on chemotaxis models with density-suppressed motility, as well as the Institut f\"ur Angewandte Mathematik, Leibniz Universit\"at Hannover, where part of this work was done, for its kind hospitality. Y.Y.~Zhang is supported by National Natural Science Foundation of China (NSFC) under the grant No. 11431005 and  Shanghai Science and Technology Committee (STCSM) under the grant No. 18dz2271000.

\bibliographystyle{siam}
\bibliography{Reference}

\end{document}